\documentclass[12pt,twoside,reqno]{amsart}
\usepackage[colorlinks=true,citecolor=blue]{hyperref}
\usepackage{mathptmx, amsmath, amssymb, amsfonts, amsthm, mathptmx, enumerate, color,mathrsfs}
\usepackage{graphicx}

\usepackage{epstopdf}

\newtheorem{theorem}{Theorem}[section]
\newtheorem{lemma}[theorem]{Lemma}
\newtheorem{proposition}[theorem]{Proposition}

\theoremstyle{definition}
\newtheorem{definition}[theorem]{Definition}

\numberwithin{equation}{section}

\begin{document}
\setcounter{page}{1}

\vspace*{2.0cm}
\title[Generalized critical Kirchhoff-type potential systems ]
{ Generalized critical Kirchhoff-type potential systems With Neumann Boundary conditions}
\author[N. Chems Eddine and M. A. Ragusa ]{Nabil CHEMS EDDINE$^{1,*}$,  Maria Alessandra Ragusa $^{2,3,*}$}
\maketitle
\vspace*{-0.6cm}

\begin{center}
{\footnotesize
$^{1}$Laboratory of Mathematical Analysis and Applications, Department of Mathematics, Faculty of Sciences,
Mohammed V University, P.O. Box 1014, Rabat, Morocco.\\
$^2$Dipartimento di Matematica e Informatica, Universitá di Catania Viale Andrea Doria, 6, 95125 Catania, Italy.\\
$^3$RUDN University, 6 Miklukho, Maklay St., 117198 Moscow, Russia.
}\end{center}

\vskip 4mm {\footnotesize \noindent {\bf Abstract.}
	
	In this paper, we consider a class of quasilinear stationary Kirchhoff type potential systems with Neumann Boundary conditions, which involves a general variable exponent elliptic operator with critical growth. Under some suitable conditions on the nonlinearities, we establish existence and multiplicity of solutions for the problem by using the concentration-compactness principle of Lions for variable exponents found in \cite{Bonder,Bonder2} and the Mountain Pass Theorem without the Palais-Smale condition given in\cite{Rabinowitz}.


 }

\renewcommand{\thefootnote}{}
\footnotetext{ $^*$Corresponding author.
\par
E-mail addresses: nab.chemseddine@gmail.com (N. Chems Eddine), mariaalessandra.ragusa@unict.it (M. A. Ragusa) \par
Received , ; Accepted . }

\section{Introduction}

The purpose of this article is to investigate the existence and multiplicity of solutions for the following class of nonlocal quasilinear elliptic systems

\begin{eqnarray}
	\label{s1.1}
	\begin{cases}
		-M_i\Big(\mathcal{A}_i(u_i)\Big)
		\textrm{div}\,\Big(  \mathcal{B}_i(\nabla u_i)\Big)=|u_i|^{s_i (x)-2}u_i+\lambda(x) F_{u_i}(x,u) & \text{in }\Omega, \\
	\mathfrak{N}.M_i\Big(\mathcal{A}_i(u_i)\Big)\mathcal{B}_i(\nabla u_i)=|u_i|^{t_i (x)-2}u_i & \text{on }\partial\Omega ;
	\end{cases}
\end{eqnarray}

for $1\leq i \leq n$  ($n\in \mathbb{N} $), where $\Omega$ is a bounded domain in $\mathbb{R}^N (N\geq 2)$, with smooth boundary $\partial \Omega$, and $\mathfrak{N}$ is the outer unit normal vector on $\partial \Omega$. $\lambda$ is continuous function,  $p_i(x), q_i(x), r_i(x),s_i(x)$ and $t_i(x)$ are Lipschitz continuous real-valued functions such that
\begin{gather}
	1 <p_i^{-}\leq p_i(x) \leq p_i^{+}< q_i^{-}\leq q_i(x) \leq q_i^{+}
	< N,
\end{gather}
and
\begin{gather}
	\gamma_i^{-}\leq \gamma_i(x) \leq \gamma_i^{+}\leq r_i^{-}\leq r_i(x) \leq r_i^{+}\leq s_i^{-}\leq s_i(x) \leq s_i^{+}\leq \gamma_i^{\star}(x)<\infty,
\end{gather}
and
\begin{gather}
	\gamma_i^{-}\leq \gamma_i(x) \leq \gamma_i^{+}\leq r_i^{-}\leq r_i(x) \leq r_i^{+}\leq t_i^{-}\leq t_i(x) \leq t_i^{+}\leq \gamma_i^{ \star\star }(x)<\infty,
\end{gather}
for all $x\in \overline{\Omega}$, where $p_i^-:= \inf_{x\in \overline{\Omega}}p_i(x)$,  $p_i^{+}:= \sup_{x\in \overline{\Omega}}p_i(x)$, and analogously to $r_i^-, r_i^+, q_i^-, q_i^+, \gamma_i^-$,
$\gamma_i^+, s_i^-, s_i^+, t_i^-$ and $t_i^+$, with $\gamma_i(x) =(1-\mathcal{H}(k^3_{i})) p_i(x)+ \mathcal{H}(k^3_{i})q_i(x)$ where $k^3_{i}$ is given in $(\textbf{\textit{H}}_2)$ and

\[
\gamma_i^{\star }(x)=\begin{cases}
	\frac{N \gamma_i(x)}{N-\gamma_i(x)} &\text{for } 	\gamma_i(x)<N ,\\
	+\infty &\text{for }	\gamma_i(x)\geq N,
\end{cases}
\text{    and   }
\gamma_i^{\star\star  }(x)=\begin{cases}
	\frac{(N-1) \gamma_i(x)}{N-\gamma_i(x)} &\text{for } 	\gamma_i(x)<N ,\\
	+\infty &\text{for }	\gamma_i(x)\geq N,
\end{cases}
\]
for all $x\in \overline{\Omega}$,  where  $\mathcal{H}:\mathbb{R}_{0}^{+}\to \left\lbrace 0,1\right\rbrace  $ is given by

\[
\mathcal{H}(k_i)=\begin{cases}
	1 & \text{ if }        	k_i>0,\\
	0 & \text{ if }            k_i<0.
\end{cases}
\]

Moreover, we consider $\mathbf{K}^1_{\gamma_{i}}:= \lbrace x \in \partial\Omega,~  t_i (x) =\gamma_i^{\star\star}(x) \rbrace$ and $\mathbf{K}^2_{\gamma_{i}}:= \lbrace x \in \overline{\Omega}, ~~s_i (x) =\gamma_i^\star(x) \rbrace$, nonempty and disjoint sets. \par

The operator $\mathcal{B}_i   : X_i\to \mathbb{R}^n$, and the operator  $ \mathcal{A}_i : X_i\to \mathbb{R}$,  are respectively defined by  $$\mathcal{B}_i(u_i)= a_i(|\nabla u_i|^{p_i(x)}) |\nabla u_i|^{p_i(x)-2}\nabla u_i, \text{ and }\mathcal{A}_i(u_i)= \displaystyle\int_{\Omega }\dfrac{1}{p_i(x)}A_i(|\nabla u_i|^{p_i(x)})dx,$$

where $X_i$ is the Banach space

$$X_i:= W^{1,p_i(x)}(\Omega)\cap W^{1,\gamma_i(x)}(\Omega),$$
$A_i(.)$ is the function $A_i(t)=\displaystyle\int_{0}^{t}a_i(k)d k$, and the function $a_i(.)$ is described in the
hypothesis $(\textbf{\textit{H}}_1)$.

In this article, we consider the function $ a_i: \mathbb{R}^+ \to \mathbb{R}^+$  satisfying the following hypotheses for all $1\leq i \leq n$ :

\begin{itemize}
	\item[$(\textbf{\textit{H}}_1)$] The function $a_i(.)$ is of   class $C^1$.
	
	\item[$(\textbf{\textit{H}}_2)$] There exist positive constants $k_i^0, k_i^1, k_i^2$ and $k_i^3$ for all $1\leq i\leq n$, such that
	
	$$k_i^0 + \mathcal{H}(k_i^3)k_i^2 \tau^{\frac{q_i(x)-p_i(x)}{p_i(x)}} \leq a_i(\tau ) \leq k_i^1 + k_i^3\tau^{\frac{q_i(x)-p_i(x)}{p_i(x)}},
	$$
	for all $\tau \geq 0$ and for almost every $x\in \overline{\Omega}.$
	
	\item[$(\textbf{\textit{H}}_3)$] There exists $c>0$ such that
	
	$$ \min \left\lbrace a_i(\tau^{p_i(x)})\tau^{p_i(x)-2}, a_i(\tau^{p_i(x)})\tau^{p_i(x)-2}
	+\tau \frac{\partial(a_i(\tau^{p_i(x)})\tau^{p_i(x)-2})}{\partial \tau } \right\rbrace \geq c\tau^{p_i(x)-2},
	$$
	for almost every $x\in \Omega$ and for all $\tau>0$.
	
	\item[$(\textbf{\textit{H}}_4)$]
	There exists positive constants $\beta_i$, $\theta_i$ and $\sigma_i$ for all $i\in \left\lbrace 1,2,...,n\right\rbrace $ such that
	
	$$ A_i(\tau )\geq \frac{1}{\beta_i}a_i(\tau)\tau \text{ with } \gamma_i^+<\theta_i <s_i^-
	\text{ and } \frac{q_i^+}{p_i^+}\leq \frac{\beta_i}{\sigma_i} < \frac{\theta_i}{p_i^+},$$
	for all $\tau \geq 0$ and $\sigma_i$ satisfy $(\mathcal{M}_2)$.
\end{itemize}

The real function $F$ belongs to $C^1(\overline{\Omega }\times \mathbb{R}^{n})$ and $F_{u_i}$ denotes the partial derivative of $F$ with respect to $u_i$.\par

$M_i: \mathbb{R}_0^+ \to \mathbb{R}^+$ is a nondecreasing and continuous function, satisfying
\begin{itemize}
	\item[$(\mathcal{M}_1)$] There exists $\mathfrak{M}_i^0>0$, such that
	$$ M_i(t)\geq \mathfrak{M}_i^0 =M_i(0),\quad  \forall t \in \mathbb{R}_{0}^+, (i=1,2,...,n).$$
	
	\item[$(\mathcal{M}_2)$] There exists $\sigma_i \in (\frac{q_i^+}{\inf\{s_i^-, t_i^-\}}, 1]$ such that
	
	$$ \widehat{M}_i \left(t\right) \geq \sigma_i M_i(t)t, \quad \forall t \in \mathbb{R}_{0}^+;$$
	
	where $\widehat{M}_i\left(t\right):= \int_{0}^{t}M_i(s)ds.$
\end{itemize}

We can see that there are many functions satisfying conditions $ (\mathcal{M}_1)-(\mathcal{M}_2)$, for example  $M(t)=\mathfrak{M}^0 + bt^{\frac{1}{\sigma}}$ with $\sigma \leq 1$, $\mathfrak{M}^0>0$ and $b\geq 0$.

The system (\ref{s1.1}) is related (in the case of a single equation)  to a model firstly proposed by Kirchhoff in 1883 as the stationary version of the Kirchhoff equation

\begin{equation}
	\label{e1.1}
	\rho\frac{\partial^2u}{\partial t^2}-\left( \dfrac{\rho_0}{h}+\dfrac{E}{2L}\int_{0}^{L}\left\vert\dfrac{\partial u(x)}{\partial x}\right\vert^2dx \right)\dfrac{\partial^2u}{\partial x^2}=0,
\end{equation}

which extends the classical D'Alembert's wave equation by considering the small vertical vibrations of a stretched elastic string when the tension is variable and the ends of the string are fixed. A distinguishing feature of equation (\ref{e1.1}) is that the equation contains a nonlocal coefficient $\dfrac{\rho_0}{h}+\dfrac{E}{2L}{\displaystyle\int_{0}^{L}\left\vert\dfrac{\partial u}{\partial x}\right\vert^2dx}$ which depends on the average $\dfrac{1}{2L}{\displaystyle\int_{0}^{L}\left\vert\dfrac{\partial u}{\partial x}\right\vert^2dx}$, and hence the equation is no longer a pointwise equation. The parameters in equation (\ref{e1.1}) possess the following meanings: $u=u(x,t)$ is the transverse string displacement at the space coordinate $x$ and time $t$,  $E$  is the Young modulus of the material (also referred to as the elastic modulus, it measures the strings resistance to being deformed elastically), $\rho$ is the mass density, $L$ is the length of the string, $h$ is the area of cross-section, and $\rho_0$ is the initial tension, see  \cite{Kir}. Almost one century later, Jacques-Louis Lions \cite{Lions} returned to the equation and proposed a general Kirchhoff equation in arbitrary dimension with external force term which was written as

\begin{eqnarray}
	\label{s1.3}
	\begin{cases}
		\frac{\partial^2u}{\partial t^2}-\left(
		a+b \int_{\Omega  }\left| \nabla u\right| ^2dx \right)\Delta u=f(x,u) &\quad \text{ in }\Omega, \\
		u=0 &\quad\text{on }\partial\Omega;
	\end{cases}
\end{eqnarray}
this problem is often called a nonlocal problem because it contains an integral over $\Omega$. This causes some mathematical difficulties which make the study of such a problem particularly interesting. The nonlocal problem models several physical and biological systems, where $u$ describes a process which depends on the average of itself, such as the population density,see \cite{Chipo} and its references therein.
For a more detailed reference on this subject we refer the interested reader to \cite{Arosio,Cavalcanti,Corr2,Heidarkhani,Miao,Perera}.\par

Moreover, the study of differential equations and variational problems driven by nonhomogeneous differential operators have been extensively investigated and received much attention because they can be presented as models for many physical phenomena. We note that the $p(x)$-Laplacian operator is a special case of the divergence form operator $$\textrm{div}\,\Big(a_i (|\nabla u_i|^{p_i(x)})|\nabla u_i|^{p_i(x)-2}\nabla u\Big),$$  for which the natural functional framework is described by the Sobolev space with variable exponent $W^{1,p_i(x)}$. We recall that, in the last two decades, particular attention has been given to variable exponent Lebesgue and Sobolev spaces, $L^{p(x)}$ and $W^{1,p(x)}$, where $p(x)$ is a real function. With the apparition of nonlinear problems in applied sciences and engineering. Lebesgue spaces $L^p$ and   Sobolev spaces $W^{1,p}$ with $p$ constant, has shown its limitations in applications. The variable exponent Lebesgue spaces and  Sobolev spaces has been used in the last decades to model phenomena concerning nonhomogeneus materials, this is, a new field research and reflects a new type of physical phenomena, for example electrorheological fluids (sometimes referred to as "smart fluid"). In these fields,  the exponent $p$ must be allowed to vary. In fact, electrorehological fluids are fluids that dramatically change their mechanical properties at the presence of an electromagnetic field, which have been used in robotics and space technology. Another field of application of these spaces is in image restoration and image processing. Moreover, other applications have emerged in thermorheological fluids, mathematical biology, flow in porous media, polycrystal plasticity, the growth of heterogeneous sand piles, and fluid dynamics, for more details see \cite{Abbas,Akdemir,D3,Gala,Goodrich,H1,Papageorgiou,Ragusa,Ru1,Ru2} and the references therein.\par

Now, we will illustrate the degree of generality of the kind of problems studied here, with adequate hypotheses on the functions $a_i$, in the following we present more some examples of problems which are also interesting from the mathematical point of view and have a wide range of applications in physics and related sciences:

\textbf{Example I.} Considering $a_i\equiv 1$, we have that $a_i$ satisfies the  $(\textbf{\textit{H}}_1),(\textbf{\textit{H}}_2)$ and $(\textbf{\textit{H}}_3)$ with $k^1_{i}=1$  and $k^2_{i}>0$ and $k^3_{i}=0$. Hence, we get the $p(x)$-Laplacian:
$$-\textrm{div}\,\Big(a_i(|\nabla u|^{p(x)})|\nabla u|^{p(x)-2}\nabla u\Big)=- \textrm{div}\,(|\nabla u|^{p(x)-2}\nabla u)= -\Delta_{p(x)}u,$$
which coincides with the usual $p$-Laplacian when $p(x)=p$, and with the Laplacian when $p(x)=2$.

\textbf{Example II.} Considering $a_i(t)= 1+ t^{\frac{q(x)-p(x)}{p(x)}}$,
 we have that $a_i$ satisfies the  $(\textbf{\textit{H}}_1),(\textbf{\textit{H}}_2)$ and $(\textbf{\textit{H}}_3)$ with $k^0_{i}=k^1_{i}=k^2_{i}=k^3_{i}=1$. Hence, we get the $p\&q$-Laplacian:

\begin{align*}
	-\textrm{div}\,\Big(a_i(|\nabla u|^{p(x)})|\nabla u|^{p(x)-2}\nabla u\Big) &= -\textrm{div}\,(|\nabla u|^{p(x)-2}\nabla u) -\textrm{div}\,(|\nabla u|^{q(x)-2}\nabla u),\\
	& = -\Delta_{p(x)}u -\Delta_{q(x)}u.
\end{align*}

This class of operators comes, for example, from a general reaction-diffusion system:
\begin{equation} \label{DCE}
u_t= \textrm{div}[D(u)\nabla u ]+ h(x,u),
\end{equation}
Where  $D(u)=|\nabla u|^{p(x)-2}+|\nabla u|^{q(x)-2}$, and the reaction term $h(x, u)$ is a polynomial of u with variable coefficients. This system has a wide range of applications in physics and related sciences, such as biophysics, plasma physics and chemical reaction design. In such applications, the function $u$ describes a concentration, the first term on the right-hand side of \ref{DCE} corresponds to the diffusion with a diffusion coefficient $D(u)$; whereas the second one is the reaction and relates to source and loss processes. Typically, in chemical and biological applications, (for further details, see \cite{Mahshid,He} references therein).

We continued with other examples that are also interesting from mathematical point of view:

\textbf{Example III.} Considering $a_i(t)= 1+ \frac{t}{\sqrt{1+t^2}}$,we have that $a_i$ satisfies the  $(\textbf{\textit{H}}_1),(\textbf{\textit{H}}_2)$ and $(\textbf{\textit{H}}_3)$ with $k^0_{i}=1$,$k^1_{i}= 2$ and $k^3_{i}=0$, $k^2_{i}>0$. Hence,we obtain the operator $p(x)$-Laplacian like or so-called the generalized Capillary operator (which is essential in applied fields like industrial, biomedical and pharmaceutical) see \cite{Ni}:

$$-\textrm{div}\,\Big(a_i(|\nabla u|^{p(x)})|\nabla u|^{p(x)-2}\nabla u\Big)= -\textrm{div}\,\Big( \Big(1+ \frac{|\nabla u|^{p(x)}}{\sqrt{1+|\nabla u|^{2p(x)}}}\Big)|\nabla u|^{p(x)-2}\nabla u\Big).$$

\textbf{Example IV.} Considering $a_i(t)= 1+ \frac{1}{(1+t)^{\frac{p(x)-2}{p(x)}}}$,we have that $a_i$ satisfies the  $(\textbf{\textit{H}}_1),(\textbf{\textit{H}}_2)$ and $(\textbf{\textit{H}}_3)$ with $k^0_{i}=1$,$k^1_{i}= 2$ and $k^3_{i}=0$, $k^2_{i}>0$. Hence,we obtain the  the generalized mean curvature operator :

$$-\textrm{div}\,\Big(a_i(|\nabla u|^{p(x)})|\nabla u|^{p(x)-2}\nabla u\Big)= -\textrm{div}\,\Big( |\nabla u|^{p(x)-2}\nabla u + \dfrac{|\nabla u|^{p(x)-2}\nabla u}{(1+ |\nabla u|^{p(x)})^{\frac{p(x)-2}{p(x)}}}\Big).$$

\textbf{Example V.} Considering $a_i(t)= 1+ t^{\frac{q(x)-p(x)}{p(x)}} +\frac{1}{(1+t)^{\frac{p(x)-2}{p(x)}}}$, we have that $a_i$ satisfies the  $(\textbf{\textit{H}}_1),(\textbf{\textit{H}}_2)$ and $(\textbf{\textit{H}}_3)$ with $k^0_{i}=1$,$k^1_{i}= 2$ and $k^3_{i}=k^2_{i}=1$. Hence,we obtain the operator

\begin{align*}
	-\textrm{div}\,\Big(a_i(|\nabla u|^{p(x)})|\nabla u|^{p(x)-2}\nabla u\Big) &= -\textrm{div}\,(|\nabla u|^{p(x)-2}\nabla u) -\textrm{div}\,(|\nabla u|^{q(x)-2}\nabla u)-\textrm{div}\,\Big( \dfrac{|\nabla u|^{p(x)-2}\nabla u}{(1+ |\nabla u|^{p(x)})^{\frac{p(x)-2}{p(x)}}}\Big),\\
	& = -\Delta_{p(x)}u -\Delta_{q(x)}u-\textrm{div}\,\Big( \dfrac{|\nabla u|^{p(x)-2}\nabla u}{(1+ |\nabla u|^{p(x)})^{\frac{p(x)-2}{p(x)}}}\Big).
\end{align*}

On the one hand, it’s well known that the class of nonlinear elliptic problems with constant critical exponents in bounded or unbounded domain occupies a considerable place in the literature, which was discussed for the first time in the seminal paper  \cite{Brezis} by Brezis and Nirenberg. Afterward, Lions \cite{Lions1} established the concentration-compactness principle in the limit case in the calculus of variation and it became one of the main techniques played an important role in order to deal with such issues. Several results have been obtained by variational methods, thus, it would be interesting to refer the reader to some works for  more related results, we refer the interested readers to \cite{Fu,Li} and references therein.\par

When $M_i\equiv1$, $a_i\equiv1$ (with $k^1_{i}=1$  and $k^2_{i}>0$ and $k^3_{i}=0$ for i=1 or 2) and $p_i\neq2$ (a constant in system of two equations),  Djellit and Tas \cite{Djellit} established the existence of nontrivial weak solutions for the systems

	\begin{eqnarray*}
	\begin{cases}
		-\Delta_p u=f(x)|u|^{p^{\star}-2}u+ \lambda F_{u}(x,u,v), & \text{in }\Omega, \\
			-\Delta_q v=g(x)|v|^{q^{\star}-2}v+ \lambda F_{v}(x,u,v), & \text{in }\Omega,\\
	  u,v \to 0, \quad\text{as } |x|\to \infty ;
	\end{cases}
\end{eqnarray*}

for all $\lambda\in (0,\lambda_1)$ by using Lions's principe with mountain pass theorem. here $\lambda_1$ is the first eigenvalue of the system
	\begin{eqnarray*}
	\begin{cases}
		-\Delta_p u=f(x)|u|^{p^{\star}-2}u+ \lambda F_{u}(x,u,v), & \text{in }\mathbb{R}^N, \\
		-\Delta_q v=g(x)|v|^{q^{\star}-2}v+ \lambda F_{v}(x,u,v), & \text{in } \mathbb{R}^N,\\
		u,v \to 0, \quad\text{as } |x|\to \infty, \quad u>0, v>0.
	\end{cases}
\end{eqnarray*}

Lalilia, Tasa and Djellit \cite{Lalilia}, inspected in detail the following system on bounded set of $\mathbb{R}^N$ with Dirichlet boundary condition

	\begin{eqnarray*}
	\begin{cases}
		-\Delta_{p(x)} u=|u|^{\alpha(x)-2}u+ \lambda F_{u}(x,u,v), & \text{in }\Omega , \\
		-\Delta_{q(x)} v=|v|^{\beta(x)-2}v+ \lambda F_{v}(x,u,v), & \text{in } \Omega,\\
		u= 0,\quad v=0,  &\text{on } \partial \Omega,
	\end{cases}
\end{eqnarray*}

where $\lambda$ is a positive parameter and $\alpha, \beta, p,q : \overline{\Omega}\to \mathbb{R} $ are Lipschitz continuous
functions verifying
 $1\leq  \alpha(x )\leq p^{\star}(x) \text{ and }1\leq  \beta (x )\leq q^{\star}(x)$,for all $x$ in $\Omega$. They proved the existence of solutions under suitable assumptions on the potential F based on variational argument.

When $M_i$ satisfying some conditions, $a_i\equiv1$ (with $k^1_{i}=1$  and $k^2_{i}>0$ and $k^3_{i}=0$),  The author \cite{chems1}, established the existence of nontrivial weak solutions for the systems

\begin{eqnarray}
	\label{}
	\begin{cases}
		-M_i\left(\displaystyle\int_{\Omega  }\dfrac{1}{p_i(x)}|\nabla u_i|^{p_i(x)}\right)
		\Delta_{p_i(x)}u_i=|u_i|^{q_i (x)-2}u_i +\lambda F_{u_i}(x,u_1,u_2,...,u_n) &\quad \text{ in }\Omega, \\
		u_i=0 &\quad\text{on }\partial\Omega;
	\end{cases}
\end{eqnarray}

for ($1\leq i\leq n$), for all  $\lambda \geq \lambda_{\star} $ by using Lions's principe with mountain pass theorem, where $\lambda_1$ isa positive constant and  $p_i,q_i : \overline{\Omega}\to \mathbb{R} $ are Lipschitz continuous
functions verifying $1\leq  q_i (x )\leq p_i^{\star}(x)$,for all $x$ in $\Omega$.\par

The author \cite{chems0}, showed the existence of infinite solutions  for a class of Kirchhoff-Type Potential Systems with critical exponent :

\begin{eqnarray*}
	\begin{cases}
		-M_i\Big(\mathcal{A}_i(u_i)\Big)
		\textrm{div}\,\Big(  \mathcal{B}_i(\nabla u_i)\Big)=|u_i|^{s_i (x)-2}u_i +\lambda F_{u_i}(x,u) & \text{in }\Omega, \\
		u=0 & \text{on }\partial\Omega;
	\end{cases}
\end{eqnarray*}

 for ($1\leq i\leq n$), where $\Omega  \subset \mathbb{R}^N$ is is a bounded domain with a smooth boundary $\partial \Omega$, $N\geq 2$, $\lambda$ is positive parameter, $p_i\in C(\overline{\Omega})$ and  $F\in C(\mathbb{R}^N\times\mathbb{R}^n,\mathbb{R})$. Existence and multiplicity results are subjected to some natural growth conditions which guarantee the Mountain Pass geometry and Palais-Smale condition.

 Our objective in this paper, is to study the existence and multiplicity of solutions for the nonlocal problem \ref{s1.1}. Precisely,  the main theorems extend in several directions previous results recently appeared in the literature, see for example \cite{Alves,chems0,chems1,Hurtado,Zhang}, and references therein. As we will see in the next sections, there are three main difficulties in our situation.  The first one  is that our problem involves the critical growth that  the lack of compactness in Sobolev embedding.  The second one comes from the appearance of the nonlocal term $M_i$, which causes some mathematical difficulties because \ref{s1.1} is no longer a pointwise identity. Finally, we can see that problem \ref{s1.1} is considered with non-standard growth conditions. This leads to the fact that the operators appeared in the problem are not homogeneous.  To overcome the above difficulties, we first use a variant of concentration-compactness principle on variable exponents Sobolev spaces extended by Bonder and Silva \cite{Bonder}, and the same principle to the variable exponent spaces from the point of view of the trace, extended by Bonder and Silva \cite{Bonder2}.
 Then, applying variational methods combined with mountain pass theorem  and the Rabinowitz’s symmetric mountain pass theorem, we obtain some existence and multiplicity results for the problems which involves a general variable exponent elliptic operator, see Theorems \ref{thm2.1} and \ref{thm2.2}.

\textit{Organization of the paper} The rest of the paper is organized as follows: in section 2 we   give some preliminary results and  state the main results and section 3 is dedicated to prove the main results.

\section{Preliminaries and basic notations}\label{sec2}

In this section, we review some preliminary  basic results regarding Lebesgue and Sobolev spaces with variable exponent. We refer the book \cite{D3}, and the papers by O. Kov\'{a}\v{c}ik and J. R\'{a}kosn\'{i} \cite{K1},and by X. Fan and D. Zhao \cite{F3}, for more detailed properties.\par

Throughout this paper we assume $p \in C(\overline{\Omega})$, $p(x)>1 $, and $\Omega $ a bounded domain of $\mathbb{R}^N$. Write
\begin{equation*}
	C_{+}(\overline{\Omega})= \{ p ;  p \in C(\overline{\Omega}) , p(x)> 1 \quad \text{ for a.e. } x \in \overline{\Omega} \} .
\end{equation*}
and

\begin{equation*}
	L_{+}^{\infty}(\Omega)= \{ p ; p \in   L^{\infty}(\Omega) , p(x)> 1 \quad \text{ for a.e. } x \in \Omega \} .
\end{equation*}

For each $ p \in C_{+}(\overline{\Omega})$ we define

\begin{gather*}
	p^{+}= \sup_{x\in \Omega}p(x) \quad\text{ and } p^{-}= \inf_{x\in \Omega}p(x) .
\end{gather*}
For any $p \in C_{+}(\overline{\Omega})$, we define the variable exponent Lebesgue space as

\begin{gather*}
	L^{p(x)}(\Omega)= \{ u : \quad \text{ u is a measurable real-valued function and } \int_{\Omega}|u(x)|^{p(x)}dx <\infty \},
\end{gather*}

endowed with the Luxemburg norm

\begin{equation*}
	|u|_{p(x)}:=|u|_{L^{p(x)}} = \inf \left\{ \mu >0 ; \int_{\Omega }\left|\frac{u(x)}{\mu}\right|^{p(x)} dx \leq 1 \right \}.
\end{equation*}

\subsection*{Remark 1} Variable exponent Lebesgue spaces resemble to classical Lebesgue spaces in many respects, the inclusions between Lebesgue spaces are naturally generalized, that is, if $0< ~mes(\Omega)<\infty$ and $p,q$ are variable exponents such that $p(x)<q(x)$ a. e. in $\Omega$, then there exists a continuous embedding
$L^{q(x)}(\Omega )\hookrightarrow L^{p(x)}( \Omega ).$

On the other hand, the variable exponent Sobolev space $W^{1,p(x)}(\Omega)$ is defined by
\[
W^{1,p(x)}(\Omega)=\{ u\in L^{p(x)}(\Omega):| \nabla
u| \in L^{p(x)}(\Omega)\},
\]
and is endoweded with the norm
\[
\| u\| _{1,p(x)}:=\| u\|_{W^{1,p(x)}(\Omega)}
=|u| _{p(x)}+| \nabla u| _{p(x)},   ~~~~ \forall u\in W^{1,p(x)}(\Omega ).
\]

It is well known that  the spaces $L^{p(x)}(\Omega)$ and $W^{1,p(x)}(\Omega )$  are separable and reflexive Banach spaces.

\begin{proposition}[H\"older Inequality,
	see \cite{D3,F3}] \label{prop1}
	The conjugate space of
	$L^{p(x)}(\Omega)$ is $L^{p'(x)}(\Omega)$, where
	\[
	\frac{1}{p(x)}+\frac{1}{p'(x)}=1.
	\]
	For any $(u,v)\in L^{p(x)}(\Omega)\times L^{p'(x)}(\Omega)$,
	we have
	\[
	| \int_{\Omega}u(x)v(x)dx|
	\leq \Big(\frac{1}{p^{-}}+ \frac{1}{(p')^{-}}\Big)| u|_{p(x)}| v| _{p'(x)}
	\leq 2| u| _{p(x)}| v|_{p'(x)}.
	\]
	Moreover, if $h_1,h_2,h_3:\overline{\Omega}\to (1,\infty)$ are Libschitz continuous functions such that
	\[
	\frac{1}{h_1(x)}+\frac{1}{h_2(x)}+\frac{1}{h_3(x)}=1,
	\]
	then for any $u \in L^{h_1(x)}(\Omega)$,$ v\in L^{h_2(x)}(\Omega)$,
	$ w\in L^{h_3(x)}(\Omega)$ the following inequality holds
	
		\[
	\int_{\Omega}| uvw|dx
	\leq \Big(\frac{1}{h_1^{-}}+ \frac{1}{h_2^{-}}+\frac{1}{h_3^{-}}\Big)| u|_{h_1(x)}| v| _{h_2(x)}| w| _{h_3(x)}.
	\]
	
\end{proposition}

\begin{proposition}[see \cite{D3,F3}] \label{prop2}
	Denote $\rho_p (u)=\int_{\Omega}| u| ^{p(x)}dx$, for all $u\in L^{p(x)}(\Omega)$. We have
	\[
	\min \{ | u| _{p(x)}^{p^{-}},| u| _{p(x)}^{p^{+}}\}
	\leq \rho_p (u)\leq \max \{ | u| _{p(x)}^{p^{-}},| u| _{p(x)
	}^{p^{+}}\}
	\]
	
	and the following implications are true
	\begin{itemize}
		\item[(i)]  $| u| _{p(x)}<1$ (resp. $=1, >1$) $\Leftrightarrow \rho_p (u)<1$
		(resp. $=1,>1$);
		
		\item[(ii)] $| u| _{p(x)}>1 \Rightarrow | u| _{p(x)}^{p^{-}}\leq \rho_p (u)
		\leq | u| _{p(x)}^{p^{+}}$;
		
		\item[(iii)] $| u| _{p(x) }<1\Rightarrow | u| _{p(x)}^{p^{+}}\leq \rho_p (u)
		\leq | u| _{p(x)}^{p^{-}}$.
		
	\end{itemize}
\end{proposition}

\begin{proposition}[see \cite{E1}] \label{prop3}
	Let $p(x)$ and $q(x)$ be measurable functions such
	that $p\in L^{\infty }(\Omega)$ and
	$1\leq p(x), q(x)\leq \infty $ almost everywhere in
	$\Omega$. If $u\in L^{q(x)}(\Omega)$, $u\neq 0$.
	Then, we have
	\begin{gather*}
		| u| _{p(x)q(x)}\leq 1\Rightarrow |u| _{p(x)q(x)}^{p^{-}}
		\leq \big|| u| ^{p(x)}\big| _{q(x)}\leq | u| _{p(x)q(x)}^{p^{+}},
		\\
		| u| _{p(x)q(x)}\geq 1\Rightarrow |u| _{p(x)q(x)}^{p^{+}}
		\leq \big| | u| ^{p(x)}\big| _{q(x)}\leq | u| _{p(x)q(x)}^{p^{-}}.
	\end{gather*}
	In particular, if $p(x)=p$ is constant, then
	\[
	| | u| ^{p}| _{q(x)}
	=|u| _{pq(x)}^{p}.
	\]
\end{proposition}

\begin{proposition}[see \cite{D3,F3}] \label{prop4}
	If $u,u_{n}\in L^{p(x)}(\Omega)$,
	$n=1,2,\dots$, then the following statements are equivalent to each other:
	\begin{itemize}
		
		\item[(1)]  $\lim_{n\to \infty } | u_{n}-u|_{p(x)}=0$,
		
		\item[(2)] $\lim_{n\to \infty } \rho_p (u_{n}-u)=0$,
		
		\item[(3)] $u_{n}\to u$ in measure in $\Omega$ and
		$\lim_{n\to \infty } \rho_p (u_{n})=\rho_p (u)$.
		
	\end{itemize}
\end{proposition}

For all $x\in \Omega $, denote by
\[
p^{\ast }(x)=\begin{cases}
	\frac{Np(x)}{N-p(x)} &\text{for } p(x)<N \\
	+\infty &\text{for }p(x)\geq N
\end{cases}
\]
the critcal Sobolev exponent of $p(x)$.

\begin{proposition}[see \cite{D3,E1}] \label{prop6}
	When $\inf_{x\in \Omega} (p^{\ast }(x)-q(x))>0$, 	we write  $q(x) \ll p^{\ast }(x)$. Then, we have\\
	(1) If $q\in L_{+}^{\infty }(\Omega)$ and
	$p(x)\leq q(x) \ll p^{\ast }(x)$, for all $x\in \mathbb{R}^N$, then the embedding
	$ W^{1,p(x)}(\Omega)\hookrightarrow L^{q(x)}(\Omega)$
	is continuous but not compact.
	
	(2) If $p$ is continuous on $\overline{\Omega }$ and $q$ is a measurable
	function on $\Omega $, with $p(x)< q(x) < p^{\ast }(x)$  for all
	$x\in\Omega $, then the embedding
	$W^{1,p(x)}(\Omega )\hookrightarrow L^{q(x)}(\Omega )$
	is compact.
\end{proposition}

Let
\begin{equation*}
q\in 	C_{+}(\partial\Omega):= \{   h \in C(\partial\Omega) , h(x)> 1 \quad \text{for all } x \in \partial\Omega \} .
\end{equation*}

The Lebesgue space $L^{q(x)}(\partial\Omega)$ is dened as
\begin{gather*}
	L^{q(x)}(\partial\Omega)= \{ u : \quad \text{ u is a measurable real-valued function and } \int_{\Omega}|u(x)|^{q(x)}dS <\infty \},
\end{gather*}

and the corresponding (Luxemburg) norm is given by

\begin{equation*}
	|u|_{	L^{q(x)}(\partial\Omega)}:=|u|_{q(x),\partial \Omega} = \inf \left\{ \mu >0 ; \int_{\Omega }\left|\frac{u(x)}{\mu}\right|^{q(x)} dS \leq 1 \right \}.
\end{equation*}

For all $x\in \partial\Omega $, denote by
\[
p^{\ast \ast  }(x)=\begin{cases}
	\frac{(N-1)p(x)}{N-p(x)} &\text{for } p(x)<N \\
	+\infty &\text{for }p(x)\geq N
\end{cases}
\]
the critcal Sobolev exponent of $p(x)$.

\begin{proposition}[see \cite{D3,E1}] \label{prop7}
	Suppose that $\Omega$ is a bounded smooth domain in $\mathbb{R}^N$, $q\in	C_{+}(\partial\Omega)$ and $p\in 	C_{+}(\overline{\Omega })$ with $N>p^+$. Then, if $q(x)\leq p^{\ast \ast }(x)$( $q(x)<p^{\ast \ast }(x)$), for all $x\in \overline{\Omega } $, the embedding
	$$W^{1,p(x)}(\Omega )\hookrightarrow L^{q(x)}(\partial \Omega ),$$
	is continuous (compact).
	
	\end{proposition}

Considering $\Gamma \subset\partial \Omega$,$\Gamma \neq \partial \Omega $ a (possibly empty) closed set, and defining

\begin{gather*}
	W^{1,p(x)}_{\Gamma}(\Omega):= \overline{\{ \phi \in C^{\infty}(\overline{\Omega }) : \quad \text{ $\phi$ vanishes in a neighborhood of $\Gamma$ }\}},
\end{gather*}

where the closure is taken in $\| .\| _{W^{1,p(x)}(\Omega)}$ norm. This is the subspace of functions vanishing
on $\Gamma$. Obviously, $W_{\emptyset}^{1,p(x)}(\Omega) =W^{1,p(x)}(\Omega)$.  In general, $W_{\Gamma}^{1,p(x)}(\Omega)=W^{1,p(x)}(\Omega)$if and
only if the $p(x)$-capacity of $\Gamma$ is zero, for more details, we refer the interested readers to \cite{Har}. The best Sobolev trace constant
$T(p(.),q(.),\Gamma)$ is defined by

$$ 0<T(p(.),q(.),\Gamma):=\inf_{v\in W^{1,p(x)}_{\Gamma}(\Omega)}
\frac{\|  v  \|_{W^{1,p(x)}(\Omega)}}{\|v \|_{L^{q(x)}(\Omega)}}$$.

Now, we recall two important versions of the concentration-compactness principle of Lions for variable exponents found in \cite{Bonder, Bonder2}, which will be used in the proof of our main results.

\begin{theorem}[see \cite{Bonder2}]\label{ccp}
	
	Let $\Omega \subset\mathbb{R}^N$ be a bounded smooth domain, $p_i\in C_{+}(\overline{\Omega})$, $t_i \in C_{+}(\partial \Omega)$ with
	$$t_i(x)\leq p_i^{ \ast \ast }(x),\quad \forall x \in \partial \Omega.$$
		
	Let $\{u_n\}_{n\in\mathbb{N}}$ be a weakly convergent sequence in
	$W^{1,p(x)}(\Omega)$ with weak limit $u$, and such that:
	\begin{itemize}
		\item $|\nabla u_n|^{p_i(x)}\rightharpoonup\mu$ weakly-*
		in the sense of measures.
		
		\item $|u_n|^{t_i(x)}\longrightarrow\nu$ weakly-* in the sense of measures.
	\end{itemize}
	Also assume  that $\mathbf{K}_1 = \{x\in \partial\Omega\colon t_i(x)=p_i^{**}(x)\}$ is
	nonempty. Then, for some countable index set $J_1$, we have:
	\begin{gather}
		\nu=|u|^{t_i(x)} + \sum_{j\in I}\nu_J\delta_{x_j}\quad \nu_j>0\\
		\mu \geq |\nabla u|^{p_i(x)} + \sum_{j\in I} \mu_j \delta_{x_j} \quad \mu_j>0\\ \label{2.1}
	\overline{ T }_{x_j} \nu_i^{1/p_i^*(x_j)} \leq \mu_i^{1/p_i(x_j)} \quad \forall j\in J_1.
	\end{gather}
	where $\{x_j\}_{j\in J_i^1}\subset \mathbf{K}_1$, $\delta_{x_j}$ is the Dirac mass at  $x_j \in \overline{\Omega}$ and $\overline{ T }_{x_j}$ is the localized Sobolev trace constant,
	namely
	\begin{equation}\label{GNS}
		\overline{ T }_{x_j} :=\sup_{\epsilon >0}
		T(p_i(.),t_i(.),\Omega_{\epsilon,j},\Lambda_{\epsilon,j}),
	\end{equation}
where $\Omega_{\epsilon,j} =\Omega \cap B_{\epsilon}(x_j)$ and $\Lambda_{\epsilon,j} =\Omega \cap \partial B_{\epsilon}(x_j)$.
	
\end{theorem}

\begin{theorem}[see \cite{Bonder}]\label{ccpp}
	Let $s_i(x)$ and $p_i(x)$ be two continuous functions such that
	$$
	1<\inf_{x\in\Omega}p_i(x)\le \sup_{x\in\Omega}p_i(x) < N \quad
	\text{and}\quad 1\le s_i(x)\le p_i^*(x)\quad \text{ in }	\overline{\Omega}.
	$$
	Let $\{u_n\}_{n\in\mathbb{N}}$ be a weakly convergent sequence in
	$W^{1,p(x)}(\Omega)$ with weak limit $u$, and such that:
	\begin{itemize}
		\item $|\nabla u_n|^{p_i(x)}\rightharpoonup\mu$ weakly-*
		in the sense of measures.
		
		\item $|u_n|^{s_i(x)}\longrightarrow\nu$ weakly-* in the sense
		of measures.
	\end{itemize}
	Also assume  that $\mathbf{K}_2= \{x\in 	\overline{\Omega}\colon s_i(x)=p_i^*(x)\}$ is
	nonempty. Then, for some countable index set $J$, we have:
	\begin{gather}
		\nu=|u|^{s_i(x)} + \sum_{j\in I}\nu_J\delta_{x_j}\quad \nu_j>0\\
		\mu \geq |\nabla u|^{p_i(x)} + \sum_{j\in I} \mu_j \delta_{x_j} \quad \mu_j>0\\ \label{2.1}
		S \nu_i^{1/p_i^*(x_j)} \leq \mu_i^{1/p_i(x_j)} \quad \forall j\in J.
	\end{gather}
	where $\{x_j\}_{j\in J_i^2}\subset \mathbf{K}_2$, $\delta_{x_j}$ is the Dirac mass at  $x_j \in \overline{\Omega}$ and $S$ is the best constant
	in the Gagliardo-Nirenberg-Sobolev inequality for variable exponents,
	namely
	\begin{equation}\label{GNSS}
		S = S_{q}(\Omega) :=\inf_{\phi\in C_0^{\infty}(\Omega)}
		\frac{\| |\nabla \phi| \|_{L^{p(x)}(\Omega)}}{\| \phi \|_{L^{q(x)}(\Omega)}}.
	\end{equation}
	
\end{theorem}

\begin{definition}
	Let $X$ be a Banach space and a functional $E\in C^1(X,\mathbb{R})$. Given sequence $(u_m)$ in $X$, if there exist $c\in \mathbb{R}$ such that
	\begin{equation}
	E(u_m) \to d   \text{ and } E'(u_m)\to 0 \text{ in } X',
	\end{equation}
	we say that $(u_m)$ is a Palais-Smale sequence with energy level $c$ ( or $(u_m)$ is $(PS)_c$ for short).
	When any $(PS)_c$ sequence for $E$ possesses some strongly convergent subsequence in $X$, we say that $E$ satisfies the Palais-Smale condition at level $c$ (or $E$ is $(PS)_c$ short).
\end{definition}
Our main tools are the classical Mountain Pass Theorem and the Rabinowitz’s $\mathbb{Z}_2$–symmetric version, that is, for even functionals, recalled
respectively in the next Theorems.
\begin{theorem}[see \cite{Rabinowitz}] \label{PMT}
	Let $X$ be a real infinite dimensional Banach space and $E\in C^{1}(X,\mathbb{R})$
	such that $E(0_X)=0$ and satisfying the (PS) condition. Suppose that
	
	\begin{itemize}
		\item[($\mathcal{I}_1$)] There are  $\mathcal{R}, \rho>0$ such that $E(u)\geq\mathcal{R}$ and for all $u\in \partial B_\rho\cap X$;
		\item[($\mathcal{I}_2$)]  There exists $e\in X$ with $\left\| e\right\| >\rho$ such that $E(e)<0$.
	\end{itemize}
	Then, $E$ possesses a critical value $c\geq \mathcal{R}$, which can be characterized as
	$$c := \inf_{\xi \in \Gamma} \sup_{t\in \left[ 0,1\right]  }E_\lambda(\xi(t)),$$
	where
	$$\Gamma= \left\lbrace  \xi : \left[ 0,1\right] \to X, \text{continuous and } \xi(0)=0_X, E(\xi (1))<0\right\rbrace. $$	
\end{theorem}

\begin{theorem}[see \cite{Rabinowitz}] \label{SPMT}
	Let $X$ be a real infinite dimensional Banach space and $E\in C^{1}(X,\mathbb{R})$ be even, satisfying the Palais-Smale condition and  $E(0_X)=0$. Suppose that condition ($\mathcal{I}_1$) holds in addition to the following:	
	\begin{itemize}
		\item[($\mathcal{I}_2 '$)]
		For each finite dimensional subspace $X_1 \subset X$, the set
		$S_1:= \{ u \in X_1: E(u) \geq 0\}$ is bounded in $X$.
	\end{itemize}
	Then $E$ has has an unbounded sequence of critical values.
\end{theorem}

In the following discussions, we will use the product space
\[
X:=\prod_{i=1}^{n}\Big(W^{1,p_i(x)}(\Omega)\cap W^{1,\gamma_i(x)}(\Omega)\Big),
\]
which is equipped with the norm
\[
\| u\|:=\max \left\lbrace  \|  u_i\|_{i} \right\rbrace ,~~~~~ \forall u=(u_1,u_2,...,u_n)\in X,
\]
where $\| u_i\|_{i}:= \| \nabla u_i\|_{p_i(x)}+ \mathcal{H}(k_i^3)\| \nabla u_i\|_{q_i(x)}$ is the norm of $W^{1,p_i(x)}(\Omega)\cap W^{1,\gamma_i(x)}(\Omega)$. The space $X^{\star}$ denotes the dual space of $X$ and equipped with the usual dual norm.

\begin{definition}
	Let $X$ be a Banach space, an element 	$u=(u_1,u_2,...,u_n) \in X$ is called a weak solution of the system  \eqref{s1.1} if 	
	\begin{align*}
		\sum_{i=1}^{n} M_i\left(
		\mathcal{A}_i(u_i)\right)\int_{\Omega  } a_i(|\nabla u_{i}| ^{p_i(x)}) |\nabla u_{i}| ^{p_i(x)-2}\nabla u_{i}\nabla v_i \,dx - \sum_{i=1}^{n}\int_{\Omega  }|u_i|^{s_i(x)-2}u_i v_i\,dx  \\
	- \sum_{i=1}^{n}\int_{\partial \Omega  }|u_i|^{t_i(x)-2}u_i v_i\,d\sigma_x	- \sum_{i=1}^{n}\displaystyle\int_{\Omega  }\lambda(x)F_{u_i}(x,u_1,...u_n)v_i\,dx =0,
	\end{align*}
	for all $v=(v_1,v_2,...,v_n)\in X= \prod_{i=1}^{n}(W^{1,p_i(x)}(\Omega)\cap W^{1,\gamma_i(x)}(\Omega))$.
\end{definition}

We denote by $E_{\lambda}$ the energy functional associated with the problem \eqref{s1.1}
$$E_{\lambda}(.):= \Phi(.)- \Theta (.) -\Upsilon(.)-\Psi_{\lambda}(.),$$
where $\Phi, \Theta \text{ and } \Psi:X\longrightarrow\mathbb{R} $
are defined as follows
\begin{align*}
	\Phi(u)&=\sum_{i=1}^{n}\widehat{M_i}\left(
	\mathcal{A}_i(u_{i}(x))\right),
	\\	
	\Theta(u)&=	\sum_{i=1}^{n}\displaystyle\int_{\Omega  }\dfrac{1}{s_i(x)}|u_i|^{s_i(x)}dx, \\
	\Upsilon(u)&=	\sum_{i=1}^{n}\displaystyle\int_{\partial\Omega  }\dfrac{1}{t_i(x)}|u_i|^{t_i(x)}d\sigma_x, \\
	\Psi_{\lambda}(u)&=\int_{\Omega}\lambda(x)F(x,u_1(x),...,u_n(x))dx,
\end{align*}
for any $u=(u_1,...,u_n)$ in $X$.\par

In order to ensure that the function $F$ satisfies the topological conditions and the geometric conditions of the mountain pass theorem (see \cite{Rabinowitz}), we assume some growth conditions.

	\subsection*{Hypotheses}
\begin{itemize}
	\item[($\mathcal{F}1$)] $\quad F\in C^{1}(\Omega \times \mathbb{R}^n,\mathbb{R})$ and $F(x,0,...,0)=0$
	
	\item[($\mathcal{F}2$)]  $\quad$ There exist positive functions $b_{ij}$ ($1\leq i,j \leq n$), such that
	\begin{gather*}
		\Big| \frac{\partial F}{\partial u_i}(x,u_1,...,u_n)\Big| \leq
		\sum_{j=1}^{n} b_{ij}(x)| u_j|^{\ell_{ij}-1},
	\end{gather*}
	where
	$1<\ell_{ij}<\inf_{x\in \Omega}\gamma_i(x)$ for all $x\in \Omega $ and for all $(i,j)\in \left\lbrace 1,2,...,n\right\rbrace^2$. The weight-functions $b_{ii}$ (resp  $b_{ij}$ if $i\neq j$) belong to the generalized Lebesgue spaces $L^{\alpha _i}(\Omega)$ (resp $L^{\alpha_{ij} }(\Omega)$), with
	\[
	\alpha _i(x)=\frac{\gamma_i(x)}{\gamma_i(x)-1}, \quad \alpha_{ij} (x)
	=\frac{\gamma_i^{\ast }(x)\gamma_j^{\ast }(x)}{\gamma_i^{\ast }(x)\gamma_j^{\ast }(x)-\gamma_i^{\ast
		}(x)-\gamma_j^{\ast }(x)}.
	\]

	\textbf{Example:}
	We give an  example of potential  $F$ satisfying hypotheses $(\mathcal{F}1)$ and $(\mathcal{F}2)$ for $n=2$. Let
	\[
	F(x,u_1,u_2)=a(x)|u_1|^{l_1(x)}|u_2|^{l_2(x)},
	\]
	where $\frac{l_1(x)}{\gamma_1(x)}+\frac{l_2(x)}{\gamma_2(x)}< 1$ and $a$ is a positive function in $L^{s(x)}(\Omega ) $ such that\\ $s(x)= \dfrac{\gamma_1^{\star}(x)\gamma_2^{\star}(x)}{\gamma_1^{\star}(x)\gamma_2^{\star}(x) -l_1(x) \gamma_2^{\star}(x)-l_2(x) \gamma_2^{\star}(x)}$ for each $x\in \Omega $.\\
	
	We can easily verify that $F(x,u_1,u_2)$ satisfies the condition $(\mathcal{F}1)$. Moreover, by using Young inequality we easily check that the condition $(\mathcal{F}2)$.\par

	\item[($\mathcal{F}3$)] There exist $K>0$ and $ \exists  \theta_i \in (\gamma_i^+, \inf\{s_i^-, t_i^-\})$ for all $(x,u_1,...,u_n)\in \Omega\times\mathbb{R}^n$ where $| u_i|^{\theta _i}\geq K $
	
	\[
	0 < F(x,u_1,...,u_n)< \sum_{i=1}^{n}  \frac{u_i}{\theta_i}\frac{\partial F}{\partial u_i}(x,u_1,...,u_n).
	\]

	\item[($\mathcal{F}4$)]	There exists $ c>0 $ such that
	\[
	|F(x,u_1,...,u_n)| \leq c\Big (\sum_{i=1}^{n} | u_i|^{r_i(x)}\Big),  \forall (x,u_1,...,u_n)\in  \Omega\times\mathbb{R}^n,
	\]
	where 	$r_i \in C_{+}(\overline{\Omega})$ and  $q_i^+< r_i^-\leq r_i^+<<  \inf\{s_i^-, t_i^-\} \leq \inf\{s_i^+, t_i^+\}\quad \forall 1\leq  i\leq n $.		
\end{itemize}

Note that according to the above hypothesis,  we have   $E_{\lambda }\in C^1(X,\mathbb{R})$ and for all $v=(v_1,v_2,...,v_n)\in  X$

\begin{align*}
	E_{\lambda }'(u)v&= \sum_{i=1}^{n} M_i\left(
	\mathcal{A}_i(u_i)\right)\int_{\Omega  } a_i(|\nabla u_{i}| ^{p_i(x)}) |\nabla u_{i}| ^{p_i(x)-2}\nabla u_{i}\nabla v_i \,dx - \sum_{i=1}^{n}\int_{\Omega  }|u_i|^{s_i(x)-2}u_i v_i\,dx \\
	&\quad - \sum_{i=1}^{n}\int_{\partial\Omega  }|u_i|^{t_i(x)-2}u_i v_i\,d\sigma_x
	- \sum_{i=1}^{n}\displaystyle\int_{\Omega} \lambda(x) F_{u_i}(x,u_1,...u_n)v_i\,dx .
\end{align*}

Consequently, $u=(u_1,u_2,...,u_n)$ in $X$ is a weak solution of \eqref{s1.1} if and only if $u$ is a critical point of $E_{\lambda }$. \par

We end this section by stating the following existence and multiplicity results.

\begin{theorem} \label{thm2.1}
	suppose that $(\mathcal{M}_1)-(\mathcal{M}_2)$ and
	$(\mathcal{F}1)-(\mathcal{F}4)$ hold. Then, there exists a constant $\lambda_{\star} >0$, such that if $\lambda(x)$ verifies
	$0<\inf_{ x\in\Omega}\lambda(x)\leq \| \lambda\|_{L^{\infty}(\Omega)}\leq \lambda_{\star}$, then
	problem \eqref{s1.1} has at least one nontrivial solution in $X$.
	
\end{theorem}

\begin{theorem} \label{thm2.2}
	Assume $(\mathcal{M}_1)-(\mathcal{M}_2)$, $(\mathcal{F}1)-(\mathcal{F}4)$, and $F(u_1,...,u_n)$ is even in $u_i$ for all $i\in \{1,2,...,n\}$.
	Then, there exists $\lambda_{\star} >0$, such that if $\lambda(x)$ verifes
	$0<\inf_{ x\in\Omega}\lambda(x)\leq \| \lambda\|_{L^{\infty}(\Omega)}\leq \lambda_{\star}$, then there exists infinitely many solutions to \eqref{s1.1} in $X$.
	
\end{theorem}


\section{The main results}\label{sec3}
To prove the main result of this paper which is given in  Theorem \ref{thm2.1}, we need to first prove few lemmas related to the mountain pass theorem and Palais-Smale condition.

\begin{lemma} \label{lemma1}
	Under the assumptions $(\mathcal{F}1)$ and $(\mathcal{F}2)$,the functional $\Psi$ is well defined, and it is of class $C^1$ on $X$. Moreover, its derivative is 	
		\[\Psi_{\lambda}^{'}(u)h=
	\sum_{i=1}^n \int_{\mathbb{R}^N} \lambda(x)\frac{\partial F}{\partial u_i}(x,u_1(x),...,u_n(x)h_i(x)\,dx, \quad \text{  } \]
	
 for all $u=(u_1,...,u_n),h=(h_1,...,h_n) \in X $.

\end{lemma}

\begin{proof} We consider $\Omega= \mathbb{R}^N$.  For all $u=(u_1,...,u_n)\in X$, under the assumptions  $(\mathcal{F}1)$ and $(\mathcal{F}2)$, we can write
	\begin{equation*}
		F(x,u_1,...,u_n)=\sum_{i=1}^n \int_{0}^{u_i} \frac{\partial F}{\partial s}(x,u_1,...,s,...,u_n)\,ds + F(x,0,...,0),
	\end{equation*}
	
	\begin{equation}
		\label{e3.4}
		F(x,u_1,...,u_n)\leq c_1\left[ \sum_{i=1}^n \left(\sum_{j=1}^n b_{ij}(x)|u_j(x)|^{\ell_{ij}-1}|u_i(x)|\right)\right].
	\end{equation}
	Then,
	\begin{align}
		\label{e3.4}
		\int_{\mathbb{R}^N} \lambda (x)F(x,u_1,...,u_n)\,dx &\leq c_2\left[ \sum_{i=1}^n \left( \int_{\mathbb{R}^N}\lambda(x)\Big(\sum_{j=1}^n b_{ij}(x)|u_j(x)|^{\ell_{ij}-1}|u_i(x)|\,dx\Big)\right)  \right]\\
		&\leq c_2\left[ \sum_{i=1}^n \|\lambda \|_{\infty}\left( \int_{\mathbb{R}^N}\Big(\sum_{j=1}^n b_{ij}(x)|u_j(x)|^{\ell_{ij}-1}|u_i(x)|\,dx\Big)\right)  \right]
	\end{align}
	If we consider the fact that $W^{1,\gamma (x)}(\mathbb{R}^N)\hookrightarrow L^{\ell(x)}(\mathbb{R}^N ),$ for $\mu(x) >1$, then there exists $c>0$ such that
	\[
	| | u| ^{\ell}| _{\gamma(x)} =|u|_{\ell \gamma (x)}^{\ell} \leq c\|u\|_{\gamma(x)}^{\ell},
	\]
	and if we apply Propositions \ref{prop1}, \ref{prop3} and \ref{prop4} and take $b_{ii} \in L^{\alpha_i(x)}$, $b_{ij} \in L^{\alpha_{ij}(x)}$ if $ i\neq j$,  then we have
	\begin{align}
		\int_{\mathbb{R}^N} F(x,u_1,...,u_n)dx &\leq c_3 \Big[  \sum_{i=1}^n
		\Big(\sum_{j=1}^n|b_{ij}|_{\alpha_{ij}(x)}||u_j|^{\ell_{ij}-1}|_{\gamma_j^{\star}(x)}|u_i|_{\gamma_i^{\star}(x)}\Big)  \Big] \\
		&\leq c_3 \Big[  \sum_{i=1}^n \Big(\sum_{j=1}^n|b_{ij}|_{\alpha_{ij}(x)}|u_j|_{(\ell_{ij}-1)\gamma_j^{\star}(x)}^{\ell_{ij}-1}|u_i|_{\gamma_i^{\star}(x)} \Big)  \Big]\\
		&\leq c_3 \Big[  \sum_{i=1}^n \Big(\sum_{j=1}^n |b_{ij}|_{\alpha_{ij}(x)}\|u_j\|_{\gamma_j(x)}^{\ell_{ij}-1}\| u_i\|_{\gamma_i(x)}\Big)  \Big] < \infty.
	\end{align}
	Hence, $\Psi_{\lambda}$ is well defined. Moreover, one can  easily see that $\Psi_{\lambda}{'}$ is also well defined on $X$. Indeed, using $(F2)$ for all $h=(h_1,...,h_n) \in X$, we have
	\begin{align*}
		\Psi_{\lambda}{'}(u)h &= \sum_{i=1}^n\int_{\mathbb{R}^N}\lambda(x)\frac{\partial F}{\partial u_i}(x,u_1,...,u_n)h_i\,dx\\
		&\leq c_4\sum_{i=1}^n \Big(\sum_{j=1}^n b_{ij}(x)|u_j(x)|^{\ell_{ij}-1}\Big)|h_i(x)|dx.
	\end{align*}
	
	Following H\"{o}lder inequality, we obtain
	\begin{align*}
		\Psi_{\lambda}{'}(u)h &\leq c_5\Big[\sum_{i=1}^n \Big(\sum_{j=1}^n
		|b_{ij}|_{\alpha_{ij}(x)}||u_j|^{\ell_{ij}-1}|_{\gamma_j^{\star}(x)}|h_i|_{{\gamma_i^{\star}(x)}}
		\Big)\Big].
	\end{align*}
	
	The above propositions yield
	\begin{align*}
		\Psi_{\lambda}{'}(u)h &\leq c\Big[\sum_{i=1}^n
		\Big(\sum_{j=1}^n
		|b_{ij}|_{\beta_{ij}(x)}\|u_j\|_{\gamma_{j}(x)}^{\ell_{ij}-1}\| h_i\|_{\gamma_i(x)}\Big)\Big] <\infty.
	\end{align*}
	
	Now let us show that $\Psi_{\lambda}$ is differentiable in the sense of Frechet, that is, for fixed $u=(u_1,...,u_n)\in X$ and given $\varepsilon>0$, there must be a $\delta=\delta_{\varepsilon,u_1,...,u_n}>0$ such that
	\begin{equation*}
		| \Psi_{\lambda}(u_1+h_1,...,u_n+h_n)- \Psi_{\lambda}(u_1,...,u_n)-\Psi_{\lambda}^{'}(u_1,...,u_n)(h_1,...,h_n)|
		\leq \varepsilon\sum_{i=1}^n(\|h_i\|_{\gamma_i(x)})
	\end{equation*}
	for all $h=(h_1,...,h_n)\in X$ with $\sum_{i=1}^n(\|h_i\|_{\gamma_i(x)}) \leq \delta$.
	
	Let $B_R$ be the ball of radius $R$ which is centered at the origin of $\mathbb{R}^N$
	and denote $B_R^{'}=\mathbb{R}^N-B_R $. Moreover, let us define the functional $\Psi_R$ on
	$\prod_{i=1}^{n}W^{1,p_{i}(x)}(B_R )\cap W^{1,\gamma_i(x)}(B_R)$ as follows:
	\begin{equation*}
		\Psi_R(u)= \int_{B_R}\lambda(x) F(x,u_1(x),...,u_n(x))\ dx.
	\end{equation*}
	If we consider $(\mathcal{F}1)$ and $(\mathcal{F}2)$, it is easy to see that $\Psi_R \in C^{1}\Big(\prod_{i=1}^{n}W^{1,p_{i}(x)}(B_R )\cap W^{1,\gamma_i(x)}(B_R)\Big)$, and in addition for all $h=(h_1,...,h_n)\in \prod_{i=1}^{n}W^{1,p_{i}(x)}(B_R )\cap W^{1,\gamma_i(x)}(B_R)$, we have
	\begin{equation*}
		\Psi_R^{'}(u)h = \sum_{i=1}^n\int_{B_R }\lambda(x)\frac{\partial F}{\partial u_i}(x,u_1(x),...,u_n(x))h_i(x)\,dx.
	\end{equation*}
	
	Also as we know, the operator $\Psi_R^{'}: X\rightarrow X^{\star}$ is compact \cite{F3}. Then, for all $u=(u_1,...,u_n), h=(h_1,...,h_n)\in X$, we can write
	
	$| \Psi_{\lambda}(u_1+h_1,...,u_n+h_n)- \Psi_{\lambda}(u_1,...,u_n)-\Psi_{\lambda}^{'}(u_1,...,u_n)(h_1,...,h_n)|$
	\[
	\leq | \Psi_R(u_1+h_1,...,u_n+h_n)- \Psi_R(u_1,...,u_n)-\Psi_R^{'}(u_1,...,u_n)(h_1,...,h_n)|
	\]
	\[
	+\left| \int_{B_R^{'} }\lambda(x)\Big(F(x,u_1+h_1,...,u_n+h_n)-F(x,u_1,...,u_n)\Big)-\right.\]
	\[\left.-\sum_{i=1}^n\int_{B_R^{'} }\lambda(x)\frac{\partial F}{\partial u_i}(x,u_1,...,u_n)h_i)\,dx \right|.
	\]
	
	According to a classical theorem, there exist $\xi_1,...,\xi_n \in ]0,1[$ such that
	\[
	\left| \int_{B_R^{'} }\lambda(x)\Big(F(x,u_1+h_1,...,u_n+h_n)-F(x,u_1,...,u_n) \Big)-\right.\]
	\[\left. -\sum_{i=1}^n\int_{B_R^{'} }\lambda(x)\frac{\partial F}{\partial u_i}(x,u_1,...,u_n)h_i\,dx \right|
	\]
	\[
	= \left| \int_{B_R^{'}}\lambda(x)\left( \sum_{i=1}^n\frac{\partial F}{\partial u_i}(x,u_1,...,u_i+\xi_i h_i,...,u_n)h_i- \sum_{i=1}^n\frac{\partial F}{\partial u_i}(x,u_1,...,u_n) \right) \, dx\right|.
	\]
	
	Using the condition $(\mathcal{F}2)$, we have
	\[
	\Big| \int_{B_R^{'} }\lambda(x) \Big(F(x,u_1+h_1,...,u_n+h_n)-F(x,u_1,...,u_n) \Big)-\sum_{i=1}^n\int_{B_R^{'} }\lambda(x)\frac{\partial F}{\partial u_i}(x,u_1,...,u_n)h_i\,dx \Big|\]
	\[\leq \Big|  \sum_{i=1}^n \Big(\sum_{j=1}^n\int_{B_R^{'}}\lambda(x)b_{ij}(x)(|u_j+\xi_j h_j|^{\ell_{ij}-1}-|u_j|^{\ell_{ij}-1})h_i dx\Big) \Big|.
	\]
	
	Using the elementary inequality $|a+b|^s\leq 2^{s-1}(|a|^s+|b|^s)$ for $a,b\in \mathbb{R}^N$, we can write
	\[\leq\sum_{i=1}^n \|\lambda\|_{\infty}\Big(\sum_{j=1}^n \Big(
	(2^{\ell_{ij}-1}-1)\int_{B_R^{'}}b_{ij}(x)|u_j|^{\ell_{ij}-1}|h_i|dx +\]\[
	+(\xi_j 2)^{\ell_{ij}-1}\int_{B_R^{'}}b_{ij}(x)|h_j|^{\ell_{ij}-1}|h_i|dx\Big)\Big).
	\]
	
	Then, applying Propositions \ref{prop1}, \ref{prop3}, and \ref{prop4},  we have
	\begin{align*}
		~~~~~&\leq\sum_{i=1}^n c\Big( \sum_{j=1}^n
		\Big(|b_{ij}(x)|_{\alpha_{ij}}\|u_j\|_{\gamma_1^{\star}(x)}^{\ell_{ij}-1}+|b_{ij}(x)|_{\alpha_{ij}}\|h_j\|_{\gamma_j^{\star}(x)}^{\ell_{ij}-1}\Big)\Big)\|h_i\|_{\gamma_i(x)},
	\end{align*}
	and by the fact that
	\begin{equation*}
		|b_{ii}(x)|_{L^{\alpha_{i}}(B_R^{'})}\longrightarrow 0,
	\end{equation*}
	\begin{equation*}
		|b_{ij}(x)|_{L^{\alpha_{ij}}(B_R^{'})}\longrightarrow 0
	\end{equation*}
	for $1\leq i,j\leq n$, as $R\rightarrow \infty$, and for $R$ sufficiently large, we obtain the estimate
	\[\left| \int_{B_R^{'} }\lambda(x)\Big(F(x,u_1+h_1,...,u_n+h_n)-F(x,u_1,...,u_n)-\right.\]
	\begin{equation*}
		\left.-\sum_{i=1}^n\frac{\partial F}{\partial u_i}(x,u_1,...,u_n)h_i\Big)\,dx \right|
		\leq \varepsilon\sum_{i=1}^n(\|h_i\|_{\gamma_i(x)}).
	\end{equation*}
	It remains only to show that $\Psi^{'}$ is continuous on $X$. Let $u^m=(u_1^m,...,u_n^m)$ be such that  $u^m\rightarrow u$ as $m\rightarrow \infty$. Then, for $h=(h_1,...,h_n)\in X$, we have
	\begin{align*}
		|\Psi_{\lambda}^{'}(u^m)h- \Psi_{\lambda}^{'}(u)h|&\leq
		|\Psi_R^{'}(u^m)h- \Psi_R^{'}(u)h|\\
		&\quad+\sum_{i=1}^n \int_{B_R^{'} }
		\Big|\lambda(x)\Big( \frac{\partial F}{\partial u_i}(x,u_1^m,...,u_n^m)h_i-\frac{\partial F}{\partial u_i}(x,u_1,...,u_n)h_i\Big)\,dx\Big|.
	\end{align*}
	
	Since $\Psi_R'$ is continuous on $\prod_{i=1}^{n}W^{1,p_{i}(x)}(B_R )\cap W^{1,\gamma_i(x)}(B_R)$(see \cite{F3}), we have
	\begin{equation*}
		|\Psi_R^{'}(u^m)h- \Psi_R^{'}(u)h|\longrightarrow 0,
	\end{equation*}
	as $m\rightarrow \infty$. Now, using $(\mathcal{F}2)$ once again and taking into account that the other terms on the right-hand side of the above inequality tend to zero, we  conclude that $\Psi_{\lambda}{'}$ is continuous on $X$. $\blacktriangleleft$
\end{proof}

\begin{lemma}
	Under the assumptions $(\mathcal{F}1)$ and $(\mathcal{F}2)$, $\Psi_{\lambda}^{'}$ is compact from $X$ to $X^{\star}$.	
\end{lemma}

\begin{proof} Let $u^m=(u_1^m,...,u_n^m)$ be a bounded sequence in $X$. Then, there exists a subsequence (we denote it also as $u^m=(u_1^m,...,u_n^m)$ ) which converges weakly in $X$ to $u=(u_1,...,u_n)\in X$. Then, if we use the same arguments as above, we have
	\begin{align*}
		|\Psi_{\lambda}^{'}(u^m)h- \Psi_{\lambda }^{'}(u)h|&\leq
		|\Psi_R^{'}(u^m)h- \Psi_R^{'}(u)h|\\
		&\quad+\sum_{i=1}^n \|\lambda\|_{\infty}\int_{B_R^{'} }
		|( \frac{\partial F}{\partial u_i}(x,u_1^m,...,u_n^m)h_i-\frac{\partial F}{\partial u_i}(x,u_1,...,u_n)h_i)\,dx|
	\end{align*}
	Since the restriction operator is continuous, we have $u^m\rightharpoonup u$ in  $\prod_{i=1}^{n}W^{1,p_{i}(x)}(B_R )\cap W^{1,\gamma_i(x)}(B_R)$. Because of the compactness of $\Psi^{'}$, the first expression on the right-hand side of the inequality tends to 0, as $ m\longrightarrow \infty$, and, as above, for sufficiently large $R$ we obtain
	$$\sum_{i=1}^n \int_{B_R^{'} }\Big|\Big( \frac{\partial F}{\partial u_i}(x,u_1^m,...,u_n^m)h_i-\frac{\partial F}{\partial u_i}(x,u_1,...,u_n)h_i\Big)\,dx\Big| \longrightarrow 0.$$
	This implies $\Psi^{'}$ is compact from $X$ to $X^{\star}$. $\blacktriangleleft$ \end{proof}

\begin{lemma} \label{lemma2}
	Let $\left\lbrace u_m=(u_{1m},u_{2m},...,u_{nm})\right\rbrace $ be a Palais-Smale sequence for the Euler-Lagrange functional $E_{\lambda }$. If $(\mathcal{F}3)$ is satisfied, then $\left\lbrace u_m\right\rbrace $ is bounded.
\end{lemma}

\begin{proof}
	Let $\left\lbrace u_m=(u_{1m},u_{2m},...,u_{nm})\right\rbrace $  be a Palais-Smale sequence for the Euler-Lagrange functional $E_{\lambda }$, we have
	\begin{align*}
		E_{\lambda }(u_m)&=\sum_{i=1}^{n}\widehat{M_i}\left(
		\mathcal{A}_i(u_{im}(x))\right)-
		\sum_{i=1}^{n}\displaystyle\int_{ \Omega  }\dfrac{1}{s_i(x)}|u_{im}(x)|^{s_i(x)}dx\\
		&\quad -
		\sum_{i=1}^{n}\displaystyle\int_{\partial \Omega  }\dfrac{1}{t_i(x)}|u_{im}(x)|^{t_i(x)}d\sigma_x - \int_{\Omega}\lambda(x)F(x,u_{1m}(x),u_{2m}(x)...,u_{nm}(x))dx\\
		&= C+o_m(1).
	\end{align*}
	On the other hand for all $v=(v_1,v_2,...,v_n)\in  X$, we have
	\begin{multline}\label{e3.1}
		E_{\lambda }'(u_m)v= \sum_{i=1}^{n} M_i\left(
		\mathcal{A}_i(u_{im})\right)\int_{\Omega  } a_i(|\nabla u_{im}| ^{p_i(x)}) |\nabla u_{im}| ^{p_i(x)-2}\nabla u_{im}\nabla v_i \,dx - \sum_{i=1}^{n}\int_{ \Omega  }|u_{im}|^{s_i(x)-2}u_{im} v_i\,dx \\
		\quad - \sum_{i=1}^{n}\int_{\partial \Omega  }|u_{im}|^{s_i(x)-2}u_{im} v_i\,d\sigma_x 	- \sum_{i=1}^{n}\displaystyle\int_{\Omega  }\lambda(x)F_{u_i}(x,u_{1m},u_{2m},...,u_{nm})v_i\,dx
		=o_m(1).
	\end{multline}
	Then
	\begin{align*}
		E_{\lambda }(u_m)-	E_{\lambda }'(u_m)\Big (\frac{u_m}{\theta }\Big)&\geq \sum_{i=1}^{n} \Big(
		\widehat{M_i}\left(
		\mathcal{A}_i(u_{im})\right) -
		\frac{1}{\theta_i}M_i\left(
		\mathcal{A}_i(u_{im})\right)\int_{\Omega  } a_i(|\nabla u_{im}| ^{p_i(x)}) |\nabla u_{im}| ^{p_i(x)}
		\Big)\\
		&\quad		+ \sum_{i=1}^{n} \Big( \frac{ 1}{\theta_i}-\frac{ 1}{s_i^-}\Big)\displaystyle\int_{\Omega  }|u_{im}|^{s_i(x)}dx	+ \sum_{i=1}^{n} \Big( \frac{ 1}{\theta_i}-\frac{ 1}{t_i^-}\Big)\displaystyle\int_{\partial\Omega  }|u_{im}|^{t_i(x)}d\sigma_x\\
		&\quad	 + \inf_{x\in \Omega}\lambda(x)  \displaystyle\int_{\Omega}
		\left[ \sum_{i=1}^{n}  \frac{u_{im}}{\theta_i}\frac{\partial F}{\partial u_i}(x,u_{1m},u_{2m},...,u_{nm})	-F(x,u_{1m},u_{2m},...,u_{nm})
		\right] dx.
	\end{align*}
	
Next, using  $(\textbf{\textit{H}}_4), (\mathcal{M}_1)-(\mathcal{M}_2)$ and  $(\mathcal{F}3)$, we obtain
	
	\begin{align*}
		E_{\lambda }(u_m)-	E_{\lambda }'(u_m)\Big (\frac{u_m}{\theta }\Big)	&\geq \sum_{i=1}^{n} \mathfrak{M}_i^0\Big(\frac{\sigma_i}{p_i^+ \beta}\int_{\Omega}A_i(|\nabla u_{im}| ^{p_i(x)})dx-
		\frac{1}{\theta_i }\int_{\Omega}a_i(|\nabla u_{im}|^{p_i(x)})|\nabla u_{im}|^{p_i(x)}dx
		\Big)\\
		&\geq \sum_{i=1}^{n}\Big(\frac{\sigma_i\mathfrak{M}_i^0}{p_i^+ \beta}- \frac{\mathfrak{M}_i^0}{\theta_i }\Big)\int_{\Omega}a_i(|\nabla u_{im}|^{p_i(x)})|\nabla u_{im}|^{p_i(x)}dx.
	\end{align*}	
	Therefore, by using $(\textbf{\textit{H}}_2)$, there are positive constants $C_{i1}$ and $C_{i2}$ such that
	
	\begin{align}\label{key}
		E_{\lambda }(u_m)-	E_{\lambda }'(u_m)\Big (\frac{u_m}{\theta }\Big)	&\geq \sum_{i=1}^{n} \Big[ C_{i1} \Big( \int_{\Omega}|\nabla u_{im}|^{p_i(x)}\Big) +C_{i2}\mathcal{H}(k_i^3)\Big(\int_{\Omega}|\nabla u_{im}|^{q_i(x)}\Big)\Big].
	\end{align}
	Suppose, by contradiction, that there exists a subsequence, still denoted by $\left\lbrace u_{im}\right\rbrace $, such that
	$\left\|  u_{im}\right\|_i
	\to + \infty $.\par
	
	If $k_i^3=0$, proposition \ref{prop2} gives us
	\begin{align*}
		E_{\lambda }(u_m)-	E_{\lambda }'(u_m)\Big (\frac{u_m}{\theta }\Big)&\geq  \sum_{i=1}^{n} C_{i}  \| u_{im}\|_{i}^{p_i^{-}},
	\end{align*}
	
	thus
	\begin{align*}
		C+ o_m(1)&\geq \sum_{i=1}^{n} C_{i}  \| u_{im}\|_{i}^{p_i^{-}},
	\end{align*}
	
	which is a contradiction because $p_i^->1$.  Thus, we conclude that $\left\lbrace u_{m}\right\rbrace $ is bounded in $X$.
	
	On the other hand, if $k_i^3>0$, we will need to analyze the following cases :\\
	
	$(i)  \left\|  u_{im} \right\| _{p_i(x)} \to +\infty$ and $\left\|  u_{im}\right\| _{q_i(x)} \to +\infty$  as $m \to  +\infty$;
	
	$(ii)  \left\|  u_{im} \right\| _{p_i(x)} \to +\infty$  and $\left\|  u_{im}\right\| _{q_i(x)}$ is bounded;

	$(iii)  \left\|  u_{im} \right\| _{p_i(x)}$ is bounded and $\left\|  u_{im}\right\| _{q_i(x)}\to +\infty$.
	
	In the case $(i)$, for m large enough, $\| u_{im}\|_{q_i(x)}^{q^-} \geq \| u_{im}\|_{q_i(x)}^{p^-} $ Hence, by \eqref{key}, we get
	
	\begin{align*}
		C+ o_m(1)&\geq \sum_{i=1}^{n}  \Big[ C_{i} \| u_{im}\|_{p_i(x)}^{p_i^{-}} +C_{i}\mathcal{H}(k_i^3)\| u_{im}\|_{q_i(x)}^{q_i^{-}}\Big],\\
		&\geq \sum_{i=1}^{n}   \Big[ C_{i} \| u_{im}\|_{p_i(x)}^{p_i^{-}} +C_{i}\mathcal{H}(k_i^3)\| u_{im}\|_{q_i(x)}^{p_i^{-}}\Big],\\
		&\geq \sum_{i=1}^{n} C_{i} \| u_{im}\|_{_i}^{p_i^{-}}
	\end{align*}
	which is absurd.

	In the case (ii), by \eqref{key}, we have
	
	\begin{align*}
		C+ o_m(1)&\geq \sum_{i=1}^{n} C_{i} \| u_{im}\|_{p_i(x)}^{p_i^{-}} ,
	\end{align*}
	
	Thence, since $p_i^->1$, taking limit as $m\to +\infty$, we obtain a contradiction.
	
	The case $(iii)$ is similar to case $(ii)$.
	
	Therefore, we conclude that $\left\lbrace u_{m}\right\rbrace $ is bounded in $X$.
	
\end{proof}

\begin{lemma} \label{lemma3}
	Let  $\left\lbrace u_m=(u_{1m},u_{2m},...,u_{nm})\right\rbrace_{m\in \mathbb{N}} \subset X$ be a Palais-Smale sequence with energy level
	$C_{\lambda} $, if
	$$C_{\lambda(x)}< \min \left\lbrace
	\inf_{1\leq i\leq n} \Big\{ \Big( \frac{1}{\theta_i}-\frac{1}{t_i^-}\Big )\overline{T_{x_j}}_{}^N\Big(D_{i}\Big)^{N/\gamma_i(x_j)}   \Big\} ,
	\inf_{1\leq i\leq n} \Big\{ \Big( \frac{1}{\theta_i}-\frac{1}{s_i^-}\Big )S_{i}^N\Big(D_{i}\Big)^{N/\gamma_i(x_j)} \Big\}  \right\rbrace , $$
		
wehre $D_i=\mathfrak{M}_i^0(k_i^0(1-\mathcal{H}(k_i^3)+\mathcal{H}(k_i^3)k_i^2 ))$. 	
	Then there exists a subsequence strongly convergent in $X$. $\overline{T_{x_j}}$ and $S_{i } $ are  the best positive constants of the Sobolev trace and  the Gagliardo-Nirenberg-Sobolev embedding, see \ref{GNS} and \ref{GNSS}.
\end{lemma}

\begin{proof}
	Let $\left\lbrace u_m\right\rbrace_{m\in \mathbb{N}} $ be a bounded Palais-Smale sequence for the functional	$E_{\lambda}$. By Lemma \ref{lemma1}, there is a subsequence still denoted by $\left\lbrace u_m\right\rbrace_{m\in \mathbb{N}} $ which converges weakly in $X$. So there exists positive and bounded measures $\mu_i$, $\nu_i \in  \Omega $ and $\overline{\nu}_i \in \partial \Omega $ such that
	
	\[
	|\nabla u_{im}|^{\gamma_i(x)}\rightharpoonup \mu_i , \quad
	| u_{im}|^{s_i(x)}\rightharpoonup \nu_i ~~~ \text{    and   } ~~| u_{im}|^{t_i(x)}\rightharpoonup \overline{\nu_i} .
	\]
	Hence by Theorem \ref{ccp} and Theorem \ref{ccpp}, if $\bigcup_{i=1}^{n}(J_i^1\cup J_i^2) =\emptyset $ then $u_{im} \rightharpoonup u_i$ in  $L^{s_i(x)}(\Omega)$ and $u_{im} \rightharpoonup u_i$ in  $ L^{t_i(x)}(\partial \Omega)$ with $i=1,2,...,n$. Let us show that if
	 $$C_{\lambda}< \min \left\lbrace
	\inf_{1\leq i\leq n} \Big\{ \Big( \frac{1}{\theta_i}-\frac{1}{t_i^-}\Big )\overline{T_{x_j}}_{}^N\Big(D_{i}\Big)^{N/\gamma_i(x_j)}   \Big\} ,
	\inf_{1\leq i\leq n} \Big\{ \Big( \frac{1}{\theta_i}-\frac{1}{s_i^-}\Big )S_{i}^N\Big(D_{i}\Big)^{N/\gamma_i(x_j)} \Big\}  \right\rbrace$$ and $\left\lbrace u_m\right\rbrace_{m\in \mathbb{N}} $ is a Palais-Smale sequence with energy level $C_\lambda$ then $J_i^1\cup J_i^2=\emptyset$ for all $ i \in \{1,2,...,n\}$. Suppose there exists $i \in \{1,2,3,...,n\}$ such that  $J_i^1\cup J_i^2$ is nonempty, then $J_i^1\neq\emptyset$ or $J_i^2\neq\emptyset$.
	
	Firstly assume the case $J_i^1\neq\emptyset$.
	Let $x_j\in \mathbf{K}^1_{\gamma_{i}}$ be a singular point of the measures $\mu_i$ and $\overline{\nu_i}$.
	 We consider $\phi \in C^{\infty}_0( \mathbb{R}^N, \left[ 0,1\right]  )$ such that $\left| \nabla \phi \right|_{\infty}\leq 2$ and
	\begin{eqnarray*}
		\label{ineq1}
		\phi(x)=
		\begin{cases}
			1,  &\quad\text{if }  \left|x\right| <1,\\
			
			0, 	&\quad\text{if } \left|x\right|\geq 2. \\
		\end{cases}
	\end{eqnarray*}
	
	We define, for any $\varepsilon>0$  and $j\in J_i^1$, the function	
	
	$$ 	 \phi_{j,\varepsilon} := \phi \Big(\frac{x-x_j}{\varepsilon}\Big),\quad \forall x \in  \mathbb{R}^N.$$
	
	Note that $\phi_{j,\varepsilon} \in C^{\infty}_0( \mathbb{R}^N, \left[ 0,1\right]  )$, $|\nabla \phi_{j,\epsilon}|_\infty\leq \frac{2}{\varepsilon}$ and
	\begin{eqnarray*}
		\label{ineq1}
		\phi_{j,\varepsilon}(x)=
		\begin{cases}
			1,  &  x\in  B_{}(x_j,\varepsilon),\\
			
			0, 	&  x\in \mathbb{R}^N\setminus B_{}(x_j,2\varepsilon). \\
		\end{cases}
	\end{eqnarray*}
	 Since $\left\lbrace u_{im}\phi_{j,\varepsilon}\right\rbrace $ is bounded in the space $W^{1,p_i(x)}(\Omega)\cap W^{1,\gamma_i(x)}(\Omega)$, it then follows from \ref{e3.1} that
	$E_{\lambda }'( u_{1m},..., u_{im},..., u_{nm})(0,...,u_{im}\phi_{j,\varepsilon},...,0
	)\rightarrow 0$ as $ m \rightarrow +\infty$, that is,
	we obtain
	\begin{align*}
		E_{\lambda }'(u_m)(0,...u_{im}\phi_{j,\varepsilon},...,0
		)&=   M_i\left(
		\mathcal{A}_i(u_{im})\right)\int_{\Omega  } a_i(|\nabla u_{im}| ^{p_i(x)}) |\nabla u_{im}| ^{p_i(x)-2}\nabla u_{im}\nabla (u_{im}\phi_{j,\varepsilon}) \,dx
		\\
		&	- \int_{\Omega  }|u_{im}|^{s_i(x)-2}u_{im} (u_{im}\phi_{j,\epsilon})\,dx
		- \int_{\partial \Omega  }|u_{im}|^{t_i(x)-2}u_{im} (u_{im}\phi_{j,\epsilon})\,d\sigma_x		\\
		&	-\int_{\Omega  }\lambda(x)  F_{u_i}(x,u_{1m},...,u_{im},...,u_{nm})u_{im}\phi_{j,\varepsilon}\,dx   \rightarrow 0 \text{ as } m \rightarrow +\infty.
	\end{align*}
	
	That is,
	\begin{multline}\label{3.2}
		M_i\left(
		\mathcal{A}_i(u_{im})\right)\int_{\Omega  } a_i(|\nabla u_{im}| ^{p_i(x)}) |\nabla u_{im}| ^{p_i(x)-2}\nabla u_{im}\nabla \phi_{j,\varepsilon} u_{im} \,dx = \int_{\Omega  }|u_{im}|^{s_i(x)}\phi_{j,\epsilon}\,dx
		+\int_{\partial \Omega  }|u_{im}|^{t_i(x)}\phi_{j,\epsilon}\,d\sigma_x
		 \\-M_i\left(
		\mathcal{A}_i(u_{im})\right)	
		\int_{\Omega  } a_i(|\nabla u_{im}| ^{p_i(x)}) |\nabla u_{im}| ^{p_i(x)}\phi_{j,\varepsilon} \,dx  + \int_{\Omega  }\lambda(x) F_{u_i}(x,u_{1m},...,u_{im},...,u_{nm})u_{im}\phi_{j,\varepsilon}\,dx +o_m(1).
	\end{multline}

Now, we will prove that

\begin{align}\label{3.3}
	\lim_{\varepsilon \to 0} \left\lbrace
	\limsup_{m\to +\infty } M_i\left(
	\mathcal{A}_i(u_{im})\right)\int_{\Omega  } a_i(|\nabla u_{im}| ^{p_i(x)}) |\nabla u_{im}| ^{p_i(x)-2}\nabla u_{im}\nabla\phi_{j,\varepsilon} u_{im} \,dx\right\rbrace  	& = 0.
\end{align}

We remark that, due to the hypotheses $(\textbf{\textit{H}}_2)$ enough to show that

\begin{align}\label{3.4}
	\lim_{\varepsilon \to 0} \left\lbrace
	\limsup_{m\to +\infty } M_i\left(
	\mathcal{A}_i(u_{im})\right)\int_{\Omega  } |\nabla u_{im}| ^{p_i(x)-2}\nabla u_{im}\nabla\phi_{j,\varepsilon} u_{im} \,dx\right\rbrace  & = 0.
\end{align}
and

\begin{align}\label{3.5}
	\lim_{\varepsilon \to 0} \left\lbrace
	\limsup_{m\to +\infty } M_i\left(
	\mathcal{A}_i(u_{im})\right)\int_{\Omega  } |\nabla u_{im}| ^{q_i(x)-2}\nabla u_{im}\nabla\phi_{j,\varepsilon} u_{im} \,dx\right\rbrace  & = 0.
\end{align}

First, using the H\"{o}lder inequality, we obtain
\begin{align*}
	\left| \int_{\Omega  }  |\nabla u_{im}| ^{p_i(x)-2}\nabla u_{im}\nabla\phi_{j,\varepsilon} u_{im}\,dx\right|
	&	\leq 2\left|  \left|\nabla u_{im}\right|^{p_i(x)-1}\right| _{\frac{p_i(x)}{p_i(x)-1}}\left| \nabla\phi_{j,\varepsilon} u_{im}\right| _{p_i(x)},
\end{align*}
since $\left\lbrace u_{im}\right\rbrace $ is bounded, the real-valued sequence $ \left|  \left|\nabla u_{im}\right|^{p_i(x)-1}\right| _{\frac{p_i(x)}{p_i(x)-1}}$ is also bounded, then there is a positive constant $C$, such that

\begin{align*}
	\left| \int_{\Omega  }  |\nabla  u_{im}| ^{p_i(x)-2}\nabla  u_{im}\nabla\phi_{j,\varepsilon}  u_{im} \,dx\right|
	&	\leq  C|\nabla\phi_{j,\varepsilon}  u_{im}|_{p_i(x)}.
\end{align*}

Moreover $\left\lbrace  u_{im}\right\rbrace $ is bounded in $W^{1,p_i(x)}(B_{}(x_j,2\varepsilon))$, then there exists a subsequence denoted again $\left\lbrace  u_{im}\right\rbrace $  weakly convergente to $u_i$ in $L^{p_i(x)}(B_{}(x_j,2\varepsilon))$. Hence

\begin{align*}
	\limsup_{m \to +\infty }	\left| \int_{\Omega  }  |\nabla u_{im}| ^{p_i(x)-2}\nabla u_{im}\nabla\phi_{j,\varepsilon} u_{im}\,dx\right|
	&	\leq  C|\nabla\phi_{j,\varepsilon} u_i|_{p_i(x)}\\
	&	\leq 2C \limsup_{\varepsilon \to 0}
	||\nabla\phi_{j,\varepsilon}|^{p_i(x)}|_{(\frac{p_i^{\star}(x)}{p_i(x)})^{'},B_{}(x_j,2\varepsilon)} ||u_i|^{p_i(x)}|_{\frac{p_i^{\star}(x)}{p_i(x)},B_{}(x_j,2\varepsilon)}	\\
	&	\leq 2C	\limsup_{\varepsilon \to 0 } ||\nabla\phi_{j,\varepsilon}|^{p_i(x)}|_{\frac{N}{p_i(x)},B_{}(x_j,2\varepsilon)} ||u_i|^{p_i(x)}|_{\frac{N}{N-p_i(x)},B_{}(x_j,2\varepsilon)}.
\end{align*}

Note that

\[
\int_{B_{}(x_j,2\varepsilon) }
(|\nabla\phi_{j,\varepsilon}|^{p_i(x)})^{(\frac{p_i^{\star}(x)}{p_i(x)})'}dx = \int_{B_{}(x_j,2\varepsilon) } |\nabla\phi_{j,\varepsilon}|^{N}dx
\leq \Big(\frac{2}{\varepsilon}\Big)^N meas (B_{}(x_j,2\varepsilon))=\frac{4^N}{N}\omega_N,
\]
where $\omega_N$ is the surface area of an $N$-dimensional unit sphere. As $ \int_{B_{}(x_j,2\varepsilon) } (|u_i|^{p_i(x)})^{\frac{p_i^{\star}(x)}{p_i(x)}}dx \to 0$ when $\varepsilon \to 0$, we obtain that $|\nabla\phi_{j,\varepsilon} u_i|_{p_i(x)}\to 0$, which implies
\begin{equation}\label{con}
	\lim_{\varepsilon \to 0} \left\lbrace
	\limsup_{n\to +\infty }\left|
	\int_{\Omega  }  |\nabla u_{im}| ^{p_i(x)-2}\nabla u_{im}\nabla\phi_{j,\epsilon} u_{im} \,dx \right| \right\rbrace = 0.
\end{equation}

Since  $\left\lbrace  u_{im}\right\rbrace $  is bounded in $W_0^{1,p_i(x)}(\Omega )$, we may assume that  $\mathcal{A}_i(u_{im}) \to t_i\geq 0 $ as $ m \to +\infty$. Observing that $M_i(t_i)$ is is continuous, we then have

\[
M_i\Big(\mathcal{A}_i(u_{im})\Big) \to M_i(t_i )\geq \mathfrak{M}_i^0>0, \quad \text{as } m \to +\infty.
\]

Hence, by \ref{con}, we obtain
\begin{equation}\label{3.7}
	\lim_{\varepsilon \to 0} \left\lbrace
	\limsup_{m\to +\infty }M_i\left(\mathcal{A}_i(u_{im})\right)
	\int_{\Omega  }  |\nabla u_{im}| ^{p_i(x)-2}\nabla u_{im}\nabla\phi_{j,\varepsilon} u_{im} \,dx\right\rbrace   = 0.
\end{equation}

\
Analogously, we verify \ref{3.5}. Therefore, we conclude the proof
of (\ref{3.3}).

Similarly, we can also get

\begin{equation}\label{3.8}
	\lim_{\varepsilon \to 0 }
	\int_{\Omega  }\lambda(x)\frac{\partial  F }{\partial u_i}(x,u_{1m},...,u_{im},...,u_{nm})\phi_{j,\epsilon} u_{im} dx=0, \text{ as } m \rightarrow +\infty.
\end{equation}

Indeed, using H\"{o}lder's inequality with $(\mathcal{F}2)$ and since $0\leq \phi_{j,\varepsilon} \leq 1$  we obtain

\begin{align*}
	\lim_{\varepsilon \to 0 }
	\int_{\Omega  }\lambda(x)\frac{\partial F }{\partial u_i}(x,u_{1m},...,u_{im},...,u_{nm})\phi_{j,\varepsilon} u_{im} d x &\leq
	\lim_{\varepsilon \to 0 }\|\lambda \|_{\infty}\int_{\Omega  }	\left( \sum_{j=1}^{n} b_{ij}(x)| u_jm|^{\ell_{ij}-1}\right) \phi_{j,\varepsilon} u_{im}dx\\
	&\leq \lim_{\varepsilon \to 0 }c\int_{\Omega  }	\left( \sum_{j=1}^{n} b_{ij}(x)| u_j|^{\ell_{ij}-1}\right)|\phi_{j,\varepsilon} u_{im}|dx\\
	&\leq \lim_{\varepsilon \to 0 }c_1 \Big(\sum_{j=1}^n
	|b_{ij}|_{\alpha_{ij}(x)}||u_{jm}|^{\ell_{ij}-1}|_{q_j^{\star}(x)}|\phi_{j,\epsilon} u_{im}|_{{q_i^{\star}(x)}}
	\Big).
\end{align*}

The above propositions yield
\begin{align*}
	\lim_{\varepsilon \to 0 }
	\int_{\Omega  }\lambda (x)\frac{\partial F }{\partial u_i}(x,u_{1m},...,u_{im},...,u_{nm})\phi_{j,\varepsilon} u_{im} d x &\leq \lim_{\varepsilon \to 0 }c_1
	\Big(\sum_{j=1}^n
	|b_{ij}|_{\beta_{ij}(x)}\|u_{jm}\|_{q_{j}(x)}^{\ell_{ij}-1}\Big)\| u_{im}\|_{q_i(x),B_{}(x_j,2\varepsilon)} .
\end{align*}
and this last goes to zero because of

$$\sum_{j=1}^n
|b_{ij}|_{\beta_{ij}(x)}\|u_j\|_{q_{j}(x)}^{\ell_{ij}-1}<\infty.$$
	
	On the other hand,
	
	$$ \lim_{\varepsilon \to 0}\int_{\Omega} \phi_{j,\epsilon}d\mu_{ij} =\mu_{ij} \phi(0) \qquad
	\text{ and }  \qquad
	\lim_{\varepsilon \to 0}\int_{\partial \Omega} \phi_{j,\epsilon}d\overline{\nu}_{ij} =\overline{\nu}_{ij} \phi(0),$$

and since $\textbf{K}_{\gamma_i}^1 \cap \textbf{K}_{\gamma_i}^2=\emptyset$, for $\epsilon>0$ sufficiently small
	$$\int_{\Omega  }|u_{im}| ^{s_i(x)}\phi_{j,\varepsilon}dx \to \int_{\Omega  }|u_{i}| ^{s_i(x)}\phi_{j,\varepsilon}dx,$$
	once that, when  $\epsilon\to 0$,
	
	$$\int_{\Omega  }|u_{i}| ^{s_i(x)}\phi_{j,\varepsilon}dx\to 0.$$

	Since $\phi_{j,\varepsilon}$ has compact support, going to the limit $m \to +\infty$ and letting $\varepsilon \to 0$ in \ref{3.2}, from \ref{3.3} and \ref{3.4},we obtain
	\begin{align}
		0 &=-\lim_{\varepsilon \to 0} \left[
		\limsup_{m\to +\infty } \Big( M_i\left(
		\mathcal{A}_i(u_{im})\right)\int_{\Omega  } a_i(|\nabla u_{im}| ^{p_i(x)}) |\nabla u_{im}| ^{p_i(x)}\phi_{j,\varepsilon}  \,dx\Big)\right]+\overline{\nu}_{ij}, \\
		&\leq -\mathfrak{M}_i^0\lim_{\varepsilon \to 0} \left[
		\limsup_{m\to +\infty } \Big(\int_{\Omega  } a_i(|\nabla u_{im}| ^{p_i(x)}) |\nabla u_{im}| ^{p_i(x)-2}\phi_{j,\varepsilon}  \,dx\Big)\right]+\overline{\nu}_{ij},\\	
		& \leq -\mathfrak{M}_i^0\lim_{\varepsilon \to 0} \left[
		\limsup_{m\to +\infty } \Big(\int_{\Omega  }
		(k_i^0 |\nabla u_{im}| ^{p_i(x)}+ \mathcal{H}(k_i^3)k_i^2  |\nabla u_{im}| ^{q_i(x)})
		\phi_{j,\varepsilon}  \,dx\Big)\right]+\overline{\nu}_{ij}.\label{inq3}	
	\end{align}
	
	Note that, when $k_i^3=0$, we have $\gamma_{i}(x)=p_i(x)$. Hence, by using (\ref{2.1}) we have 	
	\begin{align*}
		0 &\leq \overline{\nu}_{ij} - \mathfrak{M}_i^0k_i^0\lim_{\varepsilon \to 0}
		\int_{\Omega  }\phi_{j,\varepsilon}\,d\mu_i \\
		&\leq \overline{\nu}_{ij} -\mathfrak{M}_i^0k_i^0\mu_{ij} - \mathfrak{M}_i^0k_i^0\lim_{\varepsilon \to 0}
		\int_{\Omega  }|\nabla u_{i}| ^{p_i(x)}\phi_{j,\varepsilon}\,dx.
	\end{align*}	
	By using Lebesgue Dominated Convergence Theorem, we have
	$$\lim_{\varepsilon \to 0}
	\int_{\Omega  }|\nabla u_{i}| ^{p_i(x)}\phi_{j,\varepsilon}\,dx=0$$
	
	Then, we get	
	\begin{equation}\label{k0}
		\mathfrak{M}_i^0k_i^0\mu_{ij}	\leq \overline{\nu}_{ij}.
	\end{equation}

	On the other hand, if $k_i^3>0$, then $\gamma_i(x)=q_i(x)$ Therefore, follows from \eqref{GNS} and \eqref{inq3} that
	\begin{align*}
		0 &\leq \overline{\nu}_{ij} - \mathfrak{M}_i^0\lim_{\varepsilon \to 0} \left[ \limsup_{m \to 0}
		\Big(\int_{\Omega }\mathcal{H}(k_i^3)k_i^2|\nabla u_{im}| ^{q_i(x)}\phi_{j,\varepsilon}\,dx\Big) \right]  \\
		&\leq \overline{\nu}_{ij} -\mathfrak{M}_i^0\mathcal{H}(k_i^3)k_i^2\lim_{\varepsilon \to 0}
		\int_{\Omega  }\phi_{j,\varepsilon}\,d\mu_i\\
		&\leq \overline{\nu}_{ij} -\mathfrak{M}_i^0\mathcal{H}(k_i^3)k_i^2\mu_{ij} -\mathfrak{M}_i^0\mathcal{H}(k_i^3)k_i^2\lim_{\varepsilon \to 0} \int_{\Omega  }|\nabla u_{i}| ^{p_i(x)}\phi_{j,\varepsilon}\,dx,
	\end{align*}
and by using Lebesgue Dominated Convergence Theorem, we have
\begin{equation*}
	\lim_{\varepsilon \to 0}
	\int_{\Omega  }|\nabla u_{i}| ^{q_i(x)}\phi_{j,\varepsilon}\,dx=0.
\end{equation*}
	Then, we get
	\begin{equation}\label{k1}
		\mathfrak{M}_i^0\mathcal{H}(k_i^3)k_i^2\mu_{ij}	\leq \overline{\nu}_{ij}.
	\end{equation}
	Then, by combining \ref{k0} and \ref{k1}, we have
	$\mathfrak{M}_i^0((1-\mathcal{H}(k_i^3))k_i^0+ \mathcal{H}(k_i^3)k_i^2)\mu_{ij} \leq \overline{\nu}_{ij}$. Using \eqref{GNS}, we obtain 	
	
 	$$\overline{T}_{x_j}\overline{\nu}_{ij}^\frac{1}{\gamma_i^\star(x_j)}\leq \mu_{ij}^{\frac{1}{\gamma_i(x_j)}}\leq \Biggl(\frac{\overline{\nu}_{ij}}{\mathfrak{M}_i^0((1-\mathcal{H}(k_i^3))k_i^0+ \mathcal{H}(k_i^3)k_i^2)}\Biggr)^{\frac{1}{\gamma_i(x_j)}} .$$
	which implies that $\overline{\nu}_{ij}=0$ or $\overline{\nu}_{ij}\geq \overline{T_{x_j}}_{p_i}^N\Big(\mathfrak{M}_i^0(k_i^3(1-\mathcal{H}(k_i^3)+\mathcal{H}(k_i^3)k_i^2 ))\Big)^{N/\gamma_i(x_j)}$ for all $j \in J$.
	
	On the other hand, from the conditions $(\mathcal{M}_1)$, $(\mathcal{M}_2)$ and $(\mathcal{F}3)$, we get
	
	\begin{align*}
		C_{\lambda }&=E_{\lambda }(u_{1m},...,u_{im},...,u_{nm})-	E_{\lambda }'(u_{1m},...,u_{im},...,u_{nm})\Big (\frac{u_{1m}}{\theta_1 },..,\frac{u_{im}}{\theta_i },...,\frac{u_{nm}}{\theta_n }\Big)\\
		&=\sum_{i=1}^{n}\widehat{M_i}\left(\mathcal{A}_i (u_{im})\right)-
		\sum_{i=1}^{n}\displaystyle\int_{\Omega  }\dfrac{1}{s_i(x)}|u_{im}|^{s_i(x)}dx
		-\sum_{i=1}^{n}\displaystyle\int_{\partial\Omega  }\dfrac{1}{t_i(x)}|u_{im}|^{t_i(x)}d\sigma_x
		\\
		&\quad - \int_{\Omega}\lambda(x)F(x,u_{1m}(x),...,u_{nm}(x))dx-
		\sum_{i=1}^{n} M_i\left(\mathcal{A}_i (u_{im})\right)\int_{\Omega  }a_i(|\nabla u_{im}(x)| ^{p_i(x)}) \dfrac{|\nabla u_{im}(x)| ^{p_i(x)}}{\theta_i}  \,dx \\
		& \quad+ \sum_{i=1}^{n}\int_{\Omega  } \dfrac{| u_{im}(x)|^{s_i(x)}}{\theta_i}\,dx
		+ \sum_{i=1}^{n}\int_{\partial \Omega  } \dfrac{| u_{im}(x)|^{t_i(x)}}{\theta_i}\,d\sigma_x
		+ \sum_{i=1}^{n} \int_{\Omega  }\frac{\lambda(x)}{\theta_i}F_{u_i}(x,u_{1m}(x),...,u_{nm}(x))u_{im}\,dx + o_m(1),\\	
		&\geq \sum_{i=1}^{n}\frac{\sigma_i \mathfrak{M}_i^0}{p_i^+\beta_i} \int_{\Omega  }a_i(|\nabla u_{im}(x)| ^{p_i(x)})|\nabla u_{im}(x)| ^{p_i(x)}	
		- \sum_{i=1}^{n}\frac{1}{s_i^-} \int_{\Omega  }| u_{im}(x)| ^{s_i(x)}
		- \sum_{i=1}^{n}\frac{1}{t_i^-} \int_{\partial \Omega  }| u_{im}(x)| ^{t_i(x)}\\
		&\quad- \int_{\Omega}\lambda(x)F(x,u_{1m}(x),...,u_{nm}(x))dx
		- \sum_{i=1}^{n}\frac{\mathfrak{M}_i^0}{\theta_i} \int_{\Omega  }a_i(|\nabla u_{im}(x)| ^{p_i(x)})|\nabla u_{im}(x)| ^{p_i(x)}	\\
		&\quad	+\sum_{i=1}^{n}\frac{ 1}{\theta_i} \int_{\Omega  }|u_{im}(x)| ^{s_i(x)}
		+\sum_{i=1}^{n}\frac{ 1}{\theta_i} \int_{\partial \Omega  }|u_{im}(x)| ^{t_i(x)}
		+ \sum_{i=1}^{n} \int_{\Omega  }\frac{\lambda(x)}{\theta_i}F_{u_i}(x,u_{1m},...,u_{nm})u_{im}\,dx +o_m(1)\\	
		&\geq \sum_{i=1}^{n} \mathfrak{M}_i^0\Big( \frac{\sigma_i }{p_i^+\beta_i}-\frac{ 1}{\theta_i}\Big)\displaystyle\int_{\Omega  }a_i(|\nabla u_{im}(x)| ^{p_i(x)})|u_{im}|^{p_i(x)}dx
		+
		\sum_{i=1}^{n} \Big( \frac{ 1}{\theta_i}-\frac{ 1}{s_i^-}\Big)\displaystyle\int_{\Omega  }|u_{im}|^{s_i(x)}dx\\
		&\quad
			+
		\sum_{i=1}^{n} \Big( \frac{ 1}{\theta_i}-\frac{ 1}{t_i^-}\Big)\displaystyle\int_{\partial\Omega  }|u_{im}|^{t_i(x)}dx
			+ \inf_{ x\in\Omega}\lambda(x) \displaystyle\int_{\Omega}
		\left[ \sum_{i=1}^{n}  \frac{u_{im}}{\theta_i}\frac{\partial F}{\partial u_i}(x,u_{1m},...,u_{nm})	-F(x,u_{1m},...,u_{nm})
		\right] dx +o_m(1),
	\end{align*}

	hence
	
	$$ C_\lambda \geq \sum_{i=1}^{n} \Big( \frac{ 1}{\theta_i}-\frac{ 1}{t_i^-}\Big)\displaystyle\int_{\partial \Omega  }|u_{im}|^{t_i(x)}d\sigma_x + o_m(1)$$

Now, setting $\mathbf{K}^1_{i\delta} =\cup_{x\in \mathbf{K}^1_{\gamma_{i}}}(\textbf{B}_{\delta}(x)\cap \Omega)= \{ x\in \Omega : \text{dist}(x, \mathbf{K}^1_{\gamma_{i}})<\delta \}$, when $m\to +\infty$ we obtain
	
	\begin{align*}
		C_\lambda &
		\geq \sum_{i=1}^{n} \Big( \frac{ 1}{\theta_i}-\frac{ 1}{{t_i}_{\mathbf{K}^1_{i\delta}}^-}\Big)\Big( \int_{\Omega  }|u_i|^{t_i(x)}dx +\sum_{j\in J_i}\nu_{ij}\delta_{x_j}\Big)\\
		& \geq \sum_{i=1}^{n} \Big( \frac{ 1}{\theta_i}-\frac{ 1}{{t_i}_{\mathbf{K}^1_{i\delta}}^-}\Big)\Big( \int_{\Omega  }|u_i|^{s_i(x)}dx +(\overline{T_{x_j}})_{i}^N\Big(\mathfrak{M}_i^0(k_i^3(1-\mathcal{H}(k_i^3)+\mathcal{H}(k_i^3)k_i^2 ))\Big)^{N/\gamma_i(x_j)}Card J_i\Big)
	\end{align*}
Since $\delta > 0$ is arbitrary and $t_i$ is continuous, we have
	\begin{align*}
	C_\lambda &\geq \sum_{i=1}^{n} \Big( \frac{ 1}{\theta_i}-\frac{ 1}{{t_i}_{\mathbf{K}^1_{\gamma_{i}}}^-}\Big)\Big( \int_{\Omega  }|u_i|^{s_i(x)}dx +(\overline{T_{x_j}})_{i}^N\Big(\mathfrak{M}_i^0(k_i^3(1-\mathcal{H}(k_i^3)+\mathcal{H}(k_i^3)k_i^2 ))\Big)^{N/\gamma_i(x_j)}Card J_i\Big),
\end{align*}

	suppose that $\cup_{i=1}^{n}J_i^1\neq \emptyset$ and thus
	$$C_\lambda \geq  \inf_{1\leq i\leq n} \left\lbrace \Big( \frac{1}{\theta_i}-\frac{1}{{t_i}_{\mathbf{K}^1_{\gamma_{i}}}^-}\Big )\overline{T_{x_j}}_{p_i}^N\Big(\mathfrak{M}_i^0(k_i^3(1-\mathcal{H}(k_i^3)+\mathcal{H}(k_i^3)k_i^2 ))\Big)^{N/\gamma_i(x_j)} \right\rbrace.$$
	Therefore, if $C_\lambda < \inf_{1\leq i\leq n} \left\lbrace \Big( \frac{1}{\theta_i}-\frac{1}{{t_i}_{\mathbf{K}^1_{\gamma_{i}}}^-}\Big )\overline{T_{x_j}}_{p_i}^N(\mathfrak{M}_i^0)^{N/p_i(x_j)}   \right\rbrace$, the set $\cup_{i=1}^{n}J_i^1$ is embty, which means that
	$ |u_{im}|_{t_i(x)} \to |u_i|_{t_i(x)}$  for all $i=1,2,...,n$. Taking this together with the fact that $(u_{1m},...,u_{nm})\rightharpoonup (u_1,...,u_n)$
	in X, we have $u_{im} \to u_i$ strongly in $L^{t_i(x)}(\partial \Omega)$ for all $i \in \left\lbrace 1,2,...,n\right\rbrace $.

Now, considering $J_i^2\neq \emptyset$,	following the same steps as for the case $J_i^1$, we obtain
	
	$$C_{\lambda} \geq  \inf_{1\leq i\leq n} \left\lbrace \Big( \frac{1}{\theta_i}-\frac{1}{{s_i}_{\mathbf{K}^2_{\gamma_{i}}}^-}\Big )S_{p_i}^N\Big(\mathfrak{M}_i^0(k_i^3(1-\mathcal{H}(k_i^3)+\mathcal{H}(k_i^3)k_i^2 ))\Big)^{N/\gamma_i(x_j)} \right\rbrace.$$
	
Therefore $\cup_{i=1}^nJ_i^2= \emptyset$, which means that
$ |u_{im}|_{s_i(x)} \to |u_i|_{s_i(x)}$  for all $i=1,2,...,n$. Taking this
together with the fact that $(u_{1m},...,u_{nm})\rightharpoonup (u_1,...,u_n)$
in X, we have $u_{im} \to u_i$ strongly in $L^{s_i(x)}( \Omega)$ for all $i \in \left\lbrace 1,2,...,n\right\rbrace $.

	 On the other hand
	\begin{multline*}
		\left\langle E'_\lambda (u_{1m},...,u_{nm})-E'_\lambda (u_{1k},...,u_{nk}),(u_{1m}-u_{1k},0,...,0)	\right\rangle =\\
		\left\langle \Phi' (u_{1m},...,u_{nm})-\Phi' (u_{1k},...,u_{nk}),(u_{1m}-u_{1k},0,...,0)	\right\rangle\\
		-\left\langle \Theta  ' (u_{1m},...,u_{nm})-\Theta' (u_{1k},...,u_{nk}),(u_{1m}-u_{1k},0,...,0)\right\rangle\\
			-\left\langle	\Upsilon' (u_{1m},...,u_{nm})-\Upsilon' (u_{1k},...,u_{nk}),(u_{1m}-u_{1k},0,...,0)\right\rangle\\
		-\left\langle \Psi_{\lambda} ' (u_{1m},...,u_{nm})-\Psi_{\lambda}' (u_{1k},...,u_{nk}),(u_{1m}-u_{1k},0,...,0)	\right\rangle,
	\end{multline*}
	thus $E_\lambda'(u_{1m},...,u_{nm})\to 0$, i.e $E_\lambda '(u_{1m},...,u_{nm})$  is a Cauchy sequence in $X^\star $.
	Moreover, again by H\"{o}lder's inequality, we obtain
	\begin{align*}
		\left\langle \Theta ' (u_{1m},...,u_{nm})-\Theta' (u_{1k},...,u_{nk}),(u_{1m}-u_{1k},0,...,0)	\right\rangle\\
		=\int_{\Omega }\Big(|u_{1m}|^{s_1(x)-2}u_{1m}- |u_{1k}|^{s_1(x)-2}u_{1k}\Big)(u_{1m}-u_{1k})dx\\
		\leq \Vert  |u_{1m}|^{s_1(x)-2}u_{1m}- |u_{1k}|^{s_1(x)-2}u_{1k} \Vert_{L^{\frac{s_1(x)}{s_1(x)-1}}(\Omega)} \Vert u_{1m}-u_{1k}\Vert_{L^{s_1(x)}(\Omega)},
	\end{align*}
in a similar vein, we have
	\begin{align*}
	\left\langle \Upsilon ' (u_{1m},...,u_{nm})-\Upsilon' (u_{1k},...,u_{nk}),(u_{1m}-u_{1k},0,...,0)	\right\rangle\\
	=\int_{\partial\Omega }\Big(|u_{1m}|^{t_1(x)-2}u_{1m}- |u_{1k}|^{t_1(x)-2}u_{1k}\Big)(u_{1m}-u_{1k})d\sigma_x\\
	\leq \Vert  |u_{1m}|^{t_1(x)-2}u_{1m}- |u_{1k}|^{t_1(x)-2}u_{1k} \Vert_{L^{\frac{t_1(x)}{t_1(x)-1}}(\partial\Omega)} \Vert u_{1m}-u_{1k}\Vert_{L^{t_1(x)}(\partial \Omega)}.
\end{align*}
	
	Since $\left\lbrace u_{1m}\right\rbrace $ is a Cauchy sequence in $L^{s_1(x)}(\Omega)$ and in $L^{t_1(x)}(\partial \Omega)$, then  $ \Theta'(u_{1m},...,u_{nm})$ and $\Upsilon '(u_{1m},...,u_{nm})$ are Cauchy sequences
	in $X^\star$.

	 The compactness of $\Psi_{\lambda} '$ gives
	$$ (u_{1m},...,u_{nm})\rightharpoonup (u_1,...,u_n)  \Rightarrow
	\Psi_{\lambda}  '(u_{1m},...,u_{nm}) \rightarrow\Psi_{\lambda} ' (u_1,...,u_n),$$
	i.e. $\Psi_{\lambda} '(u_{1m},...,u_{nm})$ is a Cauchy sequence in $X^\star$.\par
	
	Therefore, according to the elementary inequalities (see, e.g., Auxiliary Results in \cite{Hurtado}) for any $ \varrho , \zeta \in \mathbb{R}^N $
	\begin{eqnarray}
		\label{ineq1}
		\begin{cases}
			|\varrho  -\zeta|^{p_i(x)}\leq c_{p_i} \Big( \mathcal{B}_i(\varrho)  -\mathcal{B}_i(\zeta) \Big) \cdot\left(\varrho -\zeta\right)  &\quad\text{if } p_i(x) \geq2\\
			
			|\varrho -\zeta|^{2}\leq c (|\varrho| +|\zeta| )^{2-p_i(x)}\Big( \mathcal{B}_i(\varrho)  -\mathcal{B}_i(\zeta) \Big)\cdot\left(\varrho -\zeta\right)  	&\quad\text{if } 1<p_i(x) <2 \\
		\end{cases}
	\end{eqnarray}
	where $\cdot$ denotes the standard inner product  in $\mathbb{R}^N.$ Replacing $\varrho$ and $\zeta$ by $\nabla u_{1m}$ and $\nabla u_{1k}$ respectively and integrating over $\Omega$, we obtain
	$$ c\int_{\Omega}|u_{1m} -u_{1k}|^{p_1(x)}dx \leq \left\langle  \Phi'(u_{1m},...,u_{nm})-\Phi'(u_{1k},...,u_{nk})
	,(u_{1m}-u_{1k},0,...,0)\right\rangle,$$ if $p_1(x) \geq2$, and if  $1<p_i(x) <2$, we get

	\begin{multline*}
		\int_{\Omega}\sigma_1(x)^{p_i(x)-2}|u_{1m} -u_{1k}|^{2}dx
		\leq  \left\langle \Phi'(u_{1m},...,u_{nm}) -\Phi'(u_{1k},...,u_{nk}), (u_{1m} -u_{1k},0,...,0)\right\rangle
	\end{multline*}
	
	where $\sigma_1(x) =C(|\nabla u_{1m}|+|\nabla u_{1k}|)$. Hence by  H\"{o}lder's inequality and by lemma , we get
	
	\begin{multline*}
		\int_{\Omega}|u_{1m} -u_{1k}|^{p_1(x)}dx
		= 	\int_{\Omega}\sigma_1^{\frac{p_1(x)(p_1(x)-2)}{2}}\Big(\sigma_1^{\frac{p_1(x)(p_1(x)-2)}{2}}|u_{1m} -u_{1k}|^{p_1(x)} \Big)dx	\\
		\leq C \Vert \sigma_1^{\frac{p_1(x)(2-p_1(x))}{2}} \Vert_{L^{\frac{2}{2-p_1(x)}}(\Omega)}
		\Vert \sigma_1^{\frac{p_1(x)(p_1(x)-2)}{2}}|u_{1m} -u_{1k}|^{p_1(x)} \Vert_{L^{\frac{2}{p_1(x)}}(\Omega)}\\
		\leq C\max\Biggl\{ \Vert \sigma_1\Vert^{[\frac{p_1(x)(p_1(x)-2)}{2}]^-}_{L^{p_1(x)}(\Omega)}, \Vert \sigma_1\Vert^{[\frac{p_1(x)(p_1(x)-2)}{2}]^+}_{L^{p_1(x)}(\Omega)}
		\Biggr\} \times\\
		\max\Biggl\{ \Big( 	\int_{\Omega}\sigma_1^{p_1(x)-2}|u_{1m} -u_{1k}|^{2}dx\Big)^{\frac{p_1^-}{2}}, \Big( 	\int_{\Omega}\sigma_1^{p_1(x)-2}|u_{1m} -u_{1k}|^{2}dx\Big)^{\frac{p_1^+}{2}}
		\Biggr\}.
	\end{multline*}

	Taking into account the fact that  $\left\lbrace u_{1m}\right\rbrace $  is bounded in
	$W^{1,p_1(x)}\cap W^{1,\gamma_1(x)}(\Omega)$
	$$\left\langle \Phi'(u_{1m},...,u_{nm}) -\Phi'(u_{1k},...,u_{nk}), ((u_{1m} -u_{1k},0,...,0)\right\rangle \to 0 \quad \text{as } ~m,k \to  \infty,$$
	we find that $\left\lbrace u_{1m}\right\rbrace $  is a Cauchy sequence in $W^{1,p_1(x)}\cap W^{1,\gamma_1(x)}(\Omega)$. We proceed similarly for $\left\lbrace u_{im}\right\rbrace $  with $\left\langle \Phi'(u_{1m},...,u_{im},...,u_{nm}) -\Phi'(u_{1k},...,u_{ik},...,u_{nk}), (0,...,u_{im}-u_{ik},0,...,0)\right\rangle$ for all  $i\in \left\lbrace 2,3,...,n \right\rbrace $.
\end{proof}
Now we are in position to prove Theorem \eqref{thm2.1}.

\begin{proof} (of Theorem \eqref{thm2.1})
	The proof is an immediate consequence of the mountain pass theorem, Lemma \ref{lemma2}  and Lemma \ref{lemma3}. Precisely, it suffices to verify that $E_\lambda$ has the mountain pass geometry and that $E_\lambda (tu_1,...,tu_n) <0$ for some $t>0$.
	
	From $(\mathcal{M}_2)$, we can obtain for $t>t_0$
		\begin{equation}\label{MMA}
		\widehat{M_i}(t)\leq \frac{\widehat{M_i}(t_0)}{t_0^\frac{1}{\sigma_i}}t^\frac{1}{\sigma_i}\leq c_it^\frac{1}{\sigma_i}.
	\end{equation}
		
	About the latter condition, we have
	\begin{align*}
		E_{\lambda }(u)&=\sum_{i=1}^{n}\widehat{M_i}\left(\mathcal{A}_i(u_i)\right)-\sum_{i=1}^{n}\displaystyle\int_{\Omega  }\dfrac{1}{s_i(x)}|u_i|^{s_i(x)}dx -\sum_{i=1}^{n}\displaystyle\int_{\partial \Omega  }\dfrac{1}{t_i(x)}|u_i|^{t_i(x)}d\sigma_x- \int_{\Omega}\lambda(x)F(x,u_1(x),...,u_n(x))dx,\\
		&\leq\sum_{i=1}^{n}\widehat{M_i}\left(\mathcal{A}_i(u_i)\right)-\sum_{i=1}^{n}\displaystyle\int_{\Omega  }\dfrac{1}{s_i(x)}|u_i|^{s_i(x)}dx-\sum_{i=1}^{n}\displaystyle\int_{\partial \Omega  }\dfrac{1}{t_i(x)}|u_i|^{t_i(x)}d\sigma_x-\inf_{ x\in\Omega}\lambda(x) \int_{\Omega}F(x,u_1(x),...,u_n(x))dx.
	\end{align*}
	
	Then, because of $\inf_{ x\in\Omega}\lambda(x) \Huge\int_{\Omega}F(x,u_1(x),...,u_n(x))dx >0$ and \ref{MMA}, we obtain for
	$(z_1,...,z_n) \in X/  \left\lbrace (0,...,0)\right\rbrace$ and any $t>1$
	
	\begin{align*}
		E_{\lambda }(tz_1,...,tz_n)&\leq
		\sum_{i=1}^{n}c_i\left(\mathcal{A}_i(tz_i)\right)^\frac{1}{\sigma_i} -\sum_{i=1}^{n}\displaystyle\int_{\Omega  }\dfrac{1}{s_i(x)}|tz_i|^{s_i(x)}dx-\sum_{i=1}^{n}\displaystyle\int_{\partial \Omega  }\dfrac{1}{t_i(x)}|t z_i|^{t_i(x)}d\sigma_x,\\	
		&\leq \sum_{i=1}^{n}c_i	 \left( \int_{\Omega }\dfrac{1}{p_i(x)}A_i(|\nabla (tz_i)|^{p_i(x)})dx\right)^\frac{1}{\sigma_i}-\sum_{i=1}^{n}\displaystyle\int_{\Omega  }\dfrac{1}{s_i(x)}|t z_i|^{s_i(x)}dx -\sum_{i=1}^{n}\displaystyle\int_{\partial \Omega  }\dfrac{1}{t_i(x)}|t z_i|^{t_i(x)}d\sigma_x,\\
		&\leq \sum_{i=1}^{n}
		c_i\Big(\displaystyle \int_{\Omega }\Big( \dfrac{k_i^1}{p_i(x)}|t\nabla z_i|^{p_i(x)}+
		\dfrac{k_i^3}{q_i(x)}|t\nabla z_i|^{q_i(x)}\Big)dx \Big)^\frac{1}{\sigma_i}
		-\sum_{i=1}^{n}\displaystyle\int_{\Omega  }\dfrac{1}{s_i(x)}|t z_i|^{s_i(x)}dx\\
		&~~-\sum_{i=1}^{n}\displaystyle\int_{\partial \Omega  }\dfrac{1}{t_i(x)}|t z_i|^{t_i(x)}d\sigma_x,\\
		&\leq \sum_{i=1}^{n}
		c_it^\frac{q_i^+}{\sigma_i}\Big(\displaystyle \int_{\Omega }\Big( \dfrac{k_i^1}{p_i(x)}|\nabla z_i|^{p_i(x)}+
		\dfrac{k_i^3}{q_i(x)}|\nabla z_i|^{q_i(x)}\Big)dx \Big)^\frac{1}{\sigma_i}
		-\sum_{i=1}^{n}t^{s_i^-}\displaystyle\int_{\Omega  }\dfrac{1}{s_i(x)}| z_i|^{s_i(x)}dx\\
		&\quad~~-\sum_{i=1}^{n}t^{t_i^-}\displaystyle\int_{\partial \Omega  }\dfrac{1}{t_i(x)}| z_i|^{t_i(x)}d\sigma_x.
	\end{align*}
	which tends to $-\infty$ as $t\to +\infty$ since $\sigma_i \geq \frac{q_i^+}{\inf \{s_i^-,t_i^-\} }$.

	On the other hand,  For all $u=(u_1,...,u_n)\in X$, under the assumptions $(\mathcal{M}_1)-(\mathcal{M}_2), (\textbf{\textit{H}}_2)$ and $(\mathcal{F}_4)$, we obtain
	
	\begin{align*}
		E_{\lambda}(u) &\geq \sum_{i=1}^{n}\mathfrak{M}_i \left[ 	\int_{\Omega}\frac{k_i^0}{p_i(x)} \left| \nabla u_i\right| ^{p_i(x)}dx + \mathcal{H}(k_i^3)k_i^2
		\int_{\Omega}\frac{1}{q_i(x)} \left| \nabla u_i\right| ^{q_i(x)}dx\right] \\
		&\quad -  \sum_{i=1}^{n}\int_{\Omega}\frac{1}{s_i(x)} \left| \nabla u_i\right| ^{s_i(x)}dx
		-  \sum_{i=1}^{n}\int_{\partial\Omega}\frac{1}{t_i(x)} \left| \nabla u_i\right| ^{t_i(x)}d\sigma_x -c\left\| \lambda \right\|_{\infty}  \int_{\Omega } \Big (\sum_{i=1}^{n} | u_i|^{r_i(x)}\Big) .
	\end{align*}
	Consider 	$0< \left\| u\right\| =
	\sum_{i=1}^{n}\left\| u_i\right\|_i= \rho <1$ with $\left\| u_i\right\|_i=\left\| \nabla u_i\right\|_{p_i(x)} +\mathcal{H}(k_i^3) \left\| \nabla u_i\right\|_{q_i(x)}$. By Propositons \eqref{prop1}, \eqref{prop3} and \eqref{prop4}, we have

	\begin{align*}
		E_{\lambda}(u) &\geq 	\sum_{i=1}^{n}c_i\Big(\left\| \nabla u_i\right\|_{p_i(x)}^{q^{+}} +\mathcal{H}(k_i^3)  \left\| \nabla u_i\right\|_{q_i(x)}^{q^{+}} \Big)	- \sum_{i=1}^{n}\frac{c_{1i}}{s_i^{-}}\left\|  u_i\right\|_{i}^{s_i^{-}}  	- \sum_{i=1}^{n}\frac{c_{2i}}{t_i^{-}}\left\|  u_i\right\|_{i}^{t_i^{-}}
		-\sum_{i=1}^{n}c_{3i}\left\| \lambda \right\|_{\infty} \Vert  u_i \Vert^{r_i^-}_{i}\\
		&\geq 	\sum_{i=1}^{n}\Bigl( c_i\left\|  u_i\right\|_{i}^{q_i^{+}}
		-\frac{c_{1i}}{s_i^{-}}\left\|  u_i\right\|_{i}^{s_i^{-}}
		-\frac{c_{2i}}{t_i^{-}}\left\|  u_i\right\|_{i}^{t_i^{-}} - c_{3i}\left\| \lambda \right\|_{\infty} \Vert  u_i \Vert^{r_i^-}_{i}
		\Bigr)
	\end{align*}

	Hence, since 	$q_i^{+}< r_i^{- }< \inf\{s_i^{-},t_i^{-}\}$, follows that there are $0<\rho <1$ small enough and $ \mathcal{R}>0$ such that
	
	$$E_{\lambda}(u)\geq \mathcal{R}>0 \quad \text{as} \left\| u\right\| =\rho.$$
	
	That means the existence of an element $(u_1^0,...,u_n^0)$ of $X$ such that $ E_{\lambda}(u_1^0,...,u_n^0)<0$. Consequentely, the critical value is
	$$C_\lambda := \inf_{\xi \in \Gamma} \sup_{t\in \left[ 0,1\right]  }E_\lambda(\xi(t)),$$
	where
	$$\Gamma= \left\lbrace  \xi : \left[ 0,1\right] \to X, \text{continuous and } \xi(0)=(0,...,0), \xi (1)=(u_1^0,...,u_n^0)\right\rbrace. $$
	That concludes the proof.
\end{proof}
Next we will prove under some symmetry condition on the function $F$ that \ref{s1.1} possesses infinitely many  nontrivial solutions.
\begin{proof} (of Theorem \eqref{thm2.2})
	We will use a $\mathbb{Z}_2$-symmetric version of the Mountain Pass theorem \ref{SPMT}, to accomplish the proof of theorem \ref{thm2.2}. By assumption the function $F$ is even, the functional $E_{\lambda}$ is even too. Considering the proof of theorem \ref{thm2.1}, we need only check the condition ($\mathcal{I}_2 '$). In fact by using \ref{MMA} and  $\inf_{ x\in\Omega}\lambda(x)\Huge\int_{\Omega}F(x,u_1(x),...,u_n(x))dx >0$, we obtain
	
	\begin{align*}
		E_{\lambda }(u_1,...,u_n)&\leq
		\sum_{i=1}^{n}c_i\left(\mathcal{A}_i(u_i)\right)^\frac{1}{\sigma_i} -\sum_{i=1}^{n}\displaystyle\int_{\Omega  }\dfrac{1}{s_i(x)}|u_i|^{s_i(x)}dx
		-\sum_{i=1}^{n}\displaystyle\int_{\partial \Omega  }\dfrac{1}{t_i(x)}|u_i|^{t_i(x)}d\sigma_x
		,\\	
		&\leq \sum_{i=1}^{n}
		c_i\Big(\displaystyle \int_{\Omega }\Big( \dfrac{k_i^1}{p_i^-}|\nabla u_i|^{p_i(x)}+
		\dfrac{k_i^3}{q_i^-}|\nabla u_i|^{q_i(x)}\Big)dx \Big)^\frac{1}{\sigma_i}
		-\sum_{i=1}^{n}\displaystyle\int_{\Omega  }\dfrac{1}{s_i^+}| u_i|^{s_i(x)}dx\\
&	\quad	-\sum_{i=1}^{n}\displaystyle\int_{\partial\Omega  }\dfrac{1}{t_i^+}| u_i|^{t_i(x)}dx.
	\end{align*}
	
	Let $u=(u_1,...,u_n) \in X$ be arbitrary but fixed. we define
	$$ \Omega_{<}=\left\lbrace x\in \Omega  : \left| u_i(x)\right| <1\right\rbrace, \quad \text{ } \Omega_{\geq}=\Omega  \backslash \Omega_{<},$$
	$$ \partial\Omega_{<}=\left\lbrace x\in \Omega  : \left| u_i(x)\right| <1\right\rbrace,  \text{ and } \partial \Omega_{\geq}=\partial \Omega  \backslash \Omega_{<}.$$
	
	Then we have
	
	\begin{align*}
		E_{\lambda }(u_1,...,u_n)	
		&\leq \sum_{i=1}^{n}
		c_i\Big(\displaystyle \int_{\Omega }\Big( \dfrac{k_i^1}{p_i^-}|\nabla u_i|^{p_i(x)}+
		\dfrac{k_i^3}{q_i^-}|\nabla u_i|^{q_i(x)}\Big)dx \Big)^\frac{1}{\sigma_i}
		-\sum_{i=1}^{n}\displaystyle\int_{\Omega  }\dfrac{1}{s_i^+}| u_i|^{s_i(x)}dx \\
		&	\quad	-\sum_{i=1}^{n}\displaystyle\int_{\partial\Omega  }\dfrac{1}{t_i^+}| u_i|^{t_i(x)}dx,\\
		&\leq \sum_{i=1}^{n}
		c_i\Big(\displaystyle \int_{\Omega }\Big( \dfrac{k_i^1}{p_i^-}|\nabla u_i|^{p_i(x)}+
		\dfrac{k_i^3}{q_i^-}|\nabla u_i|^{q_i(x)}\Big)dx \Big)^\frac{1}{\sigma_i}
		-\sum_{i=1}^{n}\displaystyle\int_{\Omega_{\geq}  }\dfrac{1}{s_i^+}| u_i|^{s_i^-}dx\\
		&	\quad -\sum_{i=1}^{n}\displaystyle\int_{\partial \Omega_{\geq}  }\dfrac{1}{t_i^+}| u_i|^{t_i^-}d\sigma_x,\\
		&\leq \sum_{i=1}^{n}
		c_i\Big(\displaystyle \int_{\Omega }\Big( \dfrac{k_i^1}{p_i^-}|\nabla u_i|^{p_i(x)}+
		\dfrac{k_i^3}{q_i^-}|\nabla u_i|^{q_i(x)}\Big)dx \Big)^\frac{1}{\sigma_i}-\sum_{i=1}^{n}\displaystyle\int_{\Omega}\dfrac{1}{s_i^+}| u_i|^{s_i^-}dx\\
		&\quad+\sum_{i=1}^{n}\displaystyle\int_{\Omega_{<}  }\dfrac{1}{s_i^+}| u_i|^{s_i^-}dx
		-\sum_{i=1}^{n}\displaystyle\int_{\partial \Omega}\dfrac{1}{t_i^+}| u_i|^{t_i^-}d\sigma_x
	+\sum_{i=1}^{n}\displaystyle\int_{\partial \Omega_{<}  }\dfrac{1}{t_i^+}| u_i|^{t_i^-}d\sigma_x	,\\
		&\leq \sum_{i=1}^{n}
		c_i\Big(\displaystyle \int_{\Omega }\Big( \dfrac{k_i^1}{p_i^-}|\nabla u_i|^{p_i(x)}+
		\dfrac{k_i^3}{q_i^-}|\nabla u_i|^{q_i(x)}\Big)dx \Big)^\frac{1}{\sigma_i}-\sum_{i=1}^{n}\displaystyle\int_{\Omega}\dfrac{1}{s_i^+}| u_i|^{s_i^-}dx\\
		&	\quad -\sum_{i=1}^{n}\displaystyle\int_{\partial\Omega}\dfrac{1}{t_i^+}| u_i|^{t_i^-}d\sigma_x.
	\end{align*}
	The functional $\left| .\right|_{s_i^-}: X_i\to \mathbb{R}$ defined by
	$$\left| u_i \right|_{s_i^-} = \Big(\int_{\Omega}| u_i|^{s_i^-}dx\Big)^{1/s_i^-}$$
	is a norm in $X_i$. Let $X_i^1$ be a fixed finite dimensional subspace of $X_i$. Then $\left| .\right|_{s_i^-}$ and $\left\| .\right\| _i$ are equivalent norms on  $X_i$, so there exists a positive constant $c_{1i}=c( X_i^1)$ such that
	
	$$ \left\| u_i\right\|_i^{s_i^-} \leq c_{1i}\left| u_i\right|_{s_i^-}^{s_i^-}, \text{ for all } u_i \in X_i^1,$$	
and similarly, we can see that  $ \left\| u_i\right\|_i^{t_i^-} \leq c_{2i}\left| u_i\right|_{t_i^-}^{t_i^-}, \text{ for all } u_i \in X_i^1,$
	
	Assume $\left\| u_i\right\|_i>1$  for convenience. According to proposition \ref{prop2} and proposition \ref{prop6}, for any $u\in S_1$ we obtain
	
	$$ 0\leq E_\lambda(u)\leq \sum_{i=1}^{n} \Big(c_i\left\| u_i\right\|_i^{q_i^+/\sigma_i}
	-c_{1i} \left\| u_i\right\|_i^{s_i^-}
		-c_{2i} \left\| u_i\right\|_i^{t_i^-}\Big)
	$$
	since $\inf \{s_i^-,t_i^- \}> q_i^+/\sigma_i$ we conclude that $S_1$ is bounded in $X$.

\end{proof}

\subsection*{Acknowledgements}
The secoond author wish to thank GNAMPA and the RUDN University Strategic Academic Leadership Program.

\end{document}